\documentclass[a4paper,11pt]{amsart}
\usepackage{amsfonts,stmaryrd,amssymb,amsmath,a4wide,verbatim}
\usepackage[british]{babel}
\usepackage[active]{srcltx}
\usepackage{enumitem}
\usepackage{epsfig}
\usepackage{color}

\makeatletter
\@addtoreset{equation}{section}\makeatother

\newcommand{\scal}[2]{\left\langle #1 , #2\right\rangle}

\newcommand{\normlp}[2]{\left\| #1\right\|_{#2}}

\def\diam{{\mathrm diam}\,}
\newcommand{\qedbox}{ \fbox{}}

\newenvironment{proofequidense}{\noindent\textsc{Proof of the theorem \ref{equidense}: }}{\hfill$\qedbox$}

\newenvironment{prooflipschitz}{\noindent\textsc{Proof of Theorem \ref{Lipschitz} : }}{\hfill$\qedbox$}

\newenvironment{proofchavelpinc}{\noindent\textsc{Proof of Theorem \ref{ChavelPinc}}}{\hfill$\qedbox$}
\newenvironment{prooftheoremcurvature}{\noindent\textsc{Proof Theorems \ref{distancefrontspheremodel} and \ref{Hausdorffcourbure} :}}{\hfill$\qedbox$}
\newenvironment{proofnewversion}{\noindent\textsc{Proof of Theorems \ref{distanceballmodel} and \ref{Hausdorffcourbureball} :}}{\hfill$\qedbox$}

\newtheorem{theorem}{Theorem}

\newtheorem{definition}{Definition}
\newtheorem{lemma}{Lemma}

\newtheorem{proposition}{Proposition}
\newtheorem{remark}{Remark}

\def\R{\mathbb{R}}
\def\N{\mathbb{N}}
\def\S{\mathbb{S}}
\def\B{\mathbb{B}}
\def\Vol{{\mathrm Vol}\,}
\def\H{{\mathrm H}}
\def\Bf{{\mathrm B}}
\def\Card{{\mathrm Card} }
\def\front{\mathcal F}

\def\dhn{d{\mathcal H}^n}
\def\dhnpi{d{\mathcal H}^{n+1}}

\def\Vol{{\mathrm Vol}\,}

\def\div{{\mathrm div}}
\def\dist{{\mathrm dist}}

\def\hn{{\mathcal H}^n}
\def\h1{{\mathcal H}^1}

\def\dil{{\mathrm dil}}

\begin{document}
\title[]{On the boundary of almost isoperimetric domains}

\subjclass[2000]{53A07, 53C21}

\keywords{Isoperimetric inequality}

\author[E. AUBRY, J.-F. GROSJEAN]{Erwann Aubry, Jean-Fran\c cois GROSJEAN}

\address[E. Aubry]{Université Côte d'azur, CNRS, LJAD; 28 avenue Valrose, 06108 Nice, France}
\email{eaubry@unice.fr}

\address[J.-F. Grosjean]{Institut \'Elie Cartan (Math\'ematiques), Universit\'e de Lorraine, B.P. 239, F-54506 Vand\oe uvre-les-Nancy cedex, France}
\email{jean-francois.grosjean@univ-lorraine.fr}

\date{\today}

\begin{abstract} We prove that finite perimeter subsets of $\R^{n+1}$ with small isoperimetric deficit have boundary Hausdorff-close to a sphere up to a subset of small measure. We also refine this closeness under some additional a priori integral curvature bounds. As an application, we answer a question raised by B. Colbois concerning the almost extremal hypersurfaces for Chavel's inequality.

\end{abstract}

\maketitle

\section{Introduction}
In all the paper, $B_x(r)$ and $S_x(r)$ denote respectively the Euclidean ball and sphere with center $x$ and radius $r$ in $\R^{n+1}$. We also set $\B^k$ the unit ball centred at $0$ in $\R^k$ and $\S^{k-1}$ the unit sphere centred at $0$ in $\R^k$.
\smallskip

For any Borel set $\Omega$ of $\R^{n+1}$, we denote $|\Omega|$ its Lebesgue measure, $P(\Omega)$ its perimeter (see definition in section 2) and $I(\Omega)=\frac{P(\Omega)}{|\Omega|^{\frac{n}{n+1}}}$ its isoperimetric ratio. Then it satisfies the isoperimetric inequality
\begin{equation}\label{inegisop}
I(\Omega)\geqslant I\bigl(\B^{n+1}\bigr)
\end{equation}
with equality if and only if $\Omega$ is a Euclidean ball up to set of Lebesgue measure $0$.
To study the stability of the isoperimetric inequality, we denote by
$$\delta(\Omega):=\frac{I(\Omega)}{I(\B^{n+1})}-1$$
the isoperimetric deficit of a Borel set $\Omega$ of finite perimeter and address the following question: 

\smallskip\centerline{\it"How far from a ball are almost isoperimetric domains?(i.e. with small $\delta(\Omega)$)"} 
\medskip\noindent More quantitatively, by stability of the isoperimetric inequality, we understand the validity of an inequality of the form
$$\hbox{"distance" from }\Omega\hbox{ to some ball }\leqslant C \delta(\Omega)^{1/\alpha}\hbox{ for a given category of }\Omega\subset\mathbb{R}^{n+1}$$
where the "distance" need to be defined and where $C$ and $\alpha$ are some positive universal constants.
Many authors have studied this stability problem with the Fraenkel asymmetry $\mathcal{A}(\Omega)$ as distance function. We recall that

$$\mathcal{A}(\Omega):=\inf_{x\in\R^{n+1}}\frac{|\Omega\Delta B_{x}(R_\Omega)|}{|\Omega|}\hbox{ for } \Omega\subset \R^{n+1}$$
where $R_\Omega$ is given by $R_\Omega^{n+1}|\B^{n+1}|=|\Omega|$ and $U\Delta V=(U\setminus V)\cup(V\setminus U)$. So the isoperimetric inequality is said stable with respect to the Fraenkel asymmetry if there exists $C(n),\alpha(n)>0$ such that
\begin{equation}\label{Frineq}
\mathcal{A}(\Omega)\leqslant C(n)\delta(\Omega)^{1/\alpha(n)}
\end{equation}
holds for a given category of domains $\Omega\subset \mathbb{R}^{n+1}$.

Such inequalities were first obtained for domains of $\R^2$ by Bernstein (\cite{Ber}) and Bonnesen (\cite{Bon1}). The first result in higher dimension was due to Fuglede (\cite{Fug1}) for convex domains. Without convexity assumption, the main contributions are due to Hall, Haymann, Weitsman (see \cite{Hal} and \cite{HalHayWei}) who established this inequality with $\alpha(n)=4$, and later to Fusco, Maggi and Pratelli who proved  this inequality with the sharp exponent $\alpha(n)=2$ in \cite{FusMagPrat} (see also the paper of Figalli, Magelli and Pratelli (\cite{FigMagPra}) or \cite{CicLeo} and \cite{FusJul} for other proofs of this last result).

\medskip\noindent To get more precise informations on the geometry of almost isoperimetric domains than a small Fraenkel asymmetry, we can take as "distance" function the Hausdorff distance. The first result in that direction was the following inequality proved by Bonnesen (\cite{Bon1}) for convex curves and by Fuglede (\cite{Fug2}) in the general case: if $\partial\Omega$ is a  $C^1$-piecewise closed curve there exists a Euclidean circle $\mathcal{C}$ such that
\begin{equation}\label{courbe} 16\pi d_H^2({\mathcal C},\partial\Omega)\leqslant P(\Omega)^2-4\pi |\Omega|\leqslant 4\pi |\Omega|\delta(\Omega)\bigl(2+\delta(\Omega)\bigr)
\end{equation}
where \ $d_H$\  denotes\  the\  Hausdorff\  distance. Note that assuming $\delta(\Omega)\leqslant 1$ and using\  the\  isodiametric\  inequality\  $|\Omega|\leqslant \frac{\pi}{4}(\diam\Omega)^2$, we infer the following inequality
\begin{equation}\label{courbe2} \frac{d_H({\mathcal C},\partial\Omega)}{\diam \Omega}\leqslant\frac{\sqrt{3\pi}}{4} \delta(\Omega)^\frac{1}{2}
\end{equation}
However, this result is false for more general domains in $\R^2$, especially non connected one (consider for instance the disjoint union of a large ball and a tiny one far from each other). Moreover, in higher dimension $n\geqslant 2$, even for connected smooth domains, we cannot expect  to control the Hausdorff distance from $\partial\Omega$ to a sphere by the isoperimetric deficit alone, as proves the sets obtained by adding or subtracting to a ball a thin tubular neighbourhood of a Euclidean subset of dimension  not larger than $n-1$ (see for instance \cite{Car}). So to generalize this kind of stability result in higher dimension, it is necessary to assume additional informations on the geometry of the domains we consider.
In \cite{Fug1} Fuglede proved that if $n\geqslant 3$, $\Omega$ is a convex set and $\delta(\Omega)$ small enough then
\begin{equation}\label{fuglede}
\inf_{x\in\R^{n+1}}\frac{d_H(\Omega,B_{x}(R_{\Omega}))}{R_\Omega}\leqslant C(n)\delta(\Omega)^\frac{2}{n+2}.
\end{equation}
($\delta(\Omega)^\frac{2}{n+2}$ is replaced by $\sqrt{\delta(\Omega)}$ for $n=1$ and by $(\delta(\Omega)\log[1/\delta(\Omega)])^{1/2}$ for $n=2$). Note that since $\Omega$ is convex, $\partial\Omega$ is also close to a sphere of radius $R_{\Omega}$. Actually, Fuglede deals with more general sets called nearly spherical domains and this Fuglede's result has been generalized by Fusco, Gelli and Pisante (\cite{FusGelPis}) for any set of finite perimeter satisfying an interior cone condition.

In this paper, we prove generalizations of inequalities \eqref{courbe2} and \eqref{fuglede} to any smooth domain (even nonconvex) with integral control on the mean curvature of the boundary. 
We even get a weak Hausdorff control for almost isoperimetric domains that need no additional assumption on their boundary.

\subsection{No assumption on the boundary} Let $\front(\Omega)$ be the reduced boundary of $\Omega$ (see the section 2 for the definition). When $\Omega$ is a smooth domain, we have $\front(\Omega)=\partial\Omega$.
\begin{theorem}\label{Hausdorff} Let $\Omega$ be a set of $\R^{n+1}$ with finite perimeter with $\delta(\Omega)\leqslant\frac{1}{C(n)}$. There exists $x_\Omega\in\R^{n+1}$ and $A(\Omega)\subset\front(\Omega)$ such that
\begin{enumerate} 
\item $\displaystyle\frac{\mathcal{H}^n(\front(\Omega)\setminus A(\Omega))}{P(\Omega)}\leqslant C(n)\delta(\Omega)^\frac{1}{4}$,
\item $\displaystyle\frac{d_H\bigl(A(\Omega),S_{x_\Omega}(R_{\Omega})\bigr)}{R_\Omega}\leqslant C(n)\delta(\Omega)^{\beta(n)}$.
\end{enumerate}
Here $\mathcal{H}^n$ denotes the $n$-dimensional Hausdorff measure and $\beta(n)=\min(\frac{1}{4n},\frac{1}{8})$.
\end{theorem}

\medskip\noindent In other words, the boundary  $\front(\Omega)$  is Hausdorff close to a sphere up to a set of small measure. Note that  we have
$$A(\Omega)=\front(\Omega)\cap A_{\delta(\Omega)^\frac{1}{4}}$$
where for any $\eta>0$ we set $A_{\eta}=\bigl\{x\in\R^{n+1}/\, \bigl||x-x_\Omega|-R_{\Omega}\bigr|\leqslant R_{\Omega}\eta\bigr\}$.

\begin{remark}\label{DomainHausdorff}Note that the sets of the previous theorem also satisfy
\begin{enumerate} 
\item $\displaystyle\frac{|\Omega\Delta B_{x_\Omega}(R_{\Omega})|}{|\Omega|}\leqslant C(n)\delta(\Omega)^{1/2}$ 
,
\item $\displaystyle\frac{d_H\bigl(\Omega\cap B_{x_\Omega}(R_{\Omega}),B_{x_\Omega}(R_{\Omega})\bigr)}{R_\Omega}\leqslant C(n)\delta(\Omega)^{\frac{1}{2(n+1)}}$ (see the end of the section \ref{sec:Prel}).
\end{enumerate}
In other words $\Omega\cap B_{x_\Omega}(R_\Omega)$ is Hausdorff close to the ball $B_{x_\Omega}(R_\Omega)$ up to a set of small measure, which is a weak generalization of inequality \eqref{fuglede}.
\end{remark}

\begin{remark}
When $n=1$ or $\Omega$ convex Theorem \ref{Hausdorff} easily implies  earlier results  \`{a} la Bonnesen \cite{Bon1} and Fuglede \cite{Fug1} but with non optimal power $\beta(n)$.
\end{remark}

\begin{remark}
See also Theorem \ref{Preiss} in Section \ref{ProximityPreiss} that is a reformulation of Theorem \ref{Hausdorff} in term of Preiss distance between the normalized measures associated to $\front(\Omega)$ and $S_{x_\Omega}(R_{\Omega})$.
\end{remark}

\medskip\noindent To get informations on the smooth domain $\Omega$ itself, and not up to a set of small measure, additional assumptions are required. A reasonable assumption is an integral control on the mean curvature $\H$. In the sequel, for any $p\geqslant 1$, we define 
$$\|f\|_p=\displaystyle\left(\frac{1}{P(\Omega)}\int_{\partial\Omega}|f|^p\dhn\right)^{1/p}\hbox{ for any measurable  }f:\partial\Omega\to\R.$$
\medskip\noindent Note that a upper bound on $\|\H\|_p$ with $p< n-1$ is not sufficient. Indeed,  we can refer to examples constructed by the authors in \cite{AubGro1,AubGro2}: by adding small tubular neighbourhood of well chosen trees to $B_0(1)$, we get a set almost isoperimetric domains on which $\|\H\|_p$ is uniformly  bounded for any $p<n-1$ and that is dense for the Hausdorff distance among all the closed set of $\R^{n+1}$ that contain $B_0(1)$.

\subsection{Upper bound on $\|\H\|_{n-1}$}

\begin{theorem}\label{distancefrontspheremodel} Let $\Omega$ be an open set with a smooth boundary $\partial\Omega$, finite perimeter and $\delta(\Omega)\leqslant\frac{1}{C(n)}$. There exists a subset $T$ of $\R^{n+1}$ which satisfies 
whose $1$-dimensional Hausdorff measure satisfies 
\begin{enumerate}
\item $\mathcal{H}^1(T)\leqslant C(n)R_\Omega \displaystyle\int_{\partial\Omega\setminus A_{\delta(\Omega)^{1/4}}}|\H|^{n-1}\dhn$,
\item $d_H\bigl(\partial\Omega,S_{x_\Omega}(R_{\Omega})\cup T\bigr)\leqslant C(n)R_\Omega\delta(\Omega)^{\beta(n)}$,
\item the set $A_{\delta(\Omega)^{1/4}}\cup T$ has at most $N+1$ connected components,
\end{enumerate}
where $\mathcal{H}^1(T)$ denotes the 1-dimensional Hausdorff measure of $T$ and $N$ is the number of the connected components of $\partial\Omega$ that do not intercept $A_{\delta(\Omega)^\frac{1}{4}}$.
\end{theorem}

Note that by Theorem \ref{Hausdorff} at least one connected component of $\partial\Omega$ intercepts $A_{\delta(\Omega)^\frac{1}{4}}$ and so if $\partial\Omega$ is connected then we have $N=0$ and $A_{\delta(\Omega)^{1/4}}\cup T$ is connected. Moreover note that for $n=1$ we recover Fuglede's result \eqref{courbe} for $C^2$-piecewise closed curves. 

The case $N=\infty$ in Theorem \ref{distancefrontspheremodel} is trivial since the sets obtained by the union of a sphere and infinitely numebrable many points are dense for the Hausdorff distance among all the closed sets containing $S_{x_{\Omega}}(R_\Omega)$.

 Similarly to the case of curves, Theorem \ref{distancefrontspheremodel} is quite optimal as prove examples given by a domain $\Omega_{\varepsilon}=\bigl[B_0(R)\setminus\displaystyle\bigcup_i T_{i,\varepsilon}\bigr]\cup\bigcup_jT_{j,\varepsilon}$, where $(T_i)$ and $(T_j)$ are some families of Euclidean trees and the $T_{i,\varepsilon}$ denotes the $\varepsilon$-tubular neighbourhood of $T_i$. In these examples, the integral of $|\H|^{n-1}$ on $\partial\Omega_{\varepsilon}\setminus A_{\delta(\Omega)^\frac{1}{4}}$ will converge, up to a multiplicative constant $C(n)$, to the sum of the length of the trees as $\varepsilon$ tends to $0$.
 
We refer to Theorem \ref{distanceballmodel} of Section \ref{postponed} for a generalization of inequality \eqref{fuglede} similar to Theorem \ref{distancefrontspheremodel}.

\subsection{Bound on $\|\H\|_p$ with $p> n-1$}
If we assume some upper bound on the $L^p$ norm of $|\H|$ with $p>n-1$, then combining Theorem \ref{distancefrontspheremodel} and  Lemma \ref{concentration} with H\"older inequality readily gives the  following improved result.

\begin{theorem}\label{Hausdorffcourbure} Let $p\geqslant n-1$ and $\Omega$ be an open set with a smooth boundary $\partial\Omega$, finite perimeter and $\delta(\Omega)\leqslant\frac{1}{C(n)}$. Let $(\partial\Omega_i)_{i\in I}$ be the connected components of $\partial\Omega$ that do not intercept $A_{\delta(\Omega)^\frac{1}{4}}$. For any $i\in I$, there exists $x_i\in\partial\Omega_i$ such that
\begin{align*}d_H\bigl(\partial\Omega,& S_{x_\Omega}(R_{\Omega})\cup\bigcup_{i\in I}\{x_i\}\bigr)\leqslant C(n,p)R_\Omega\left[\delta(\Omega)^{\beta(n)}+\delta(\Omega)^{\frac{p-n+1}{4p}}(P(\Omega)^{\frac{1}{n}}\|\H\|_p)^{n-1}\right]
\end{align*}
Moreover if $p\geqslant n$ and if $\H$ is $L^p$-integrable then $I$ is finite and we have
\begin{equation}\label{cardinal}
\Card(I)\leqslant C(n,p)P(\Omega)\|\H\|_p^n\delta(\Omega)^\frac{p-n}{4p}.
\end{equation}
\end{theorem}
\begin{remark} We will see in the proof that the above  estimates are more precise since as in Theorem \ref{distancefrontspheremodel}, we can replaced $\|\H\|_ p$ by $\Bigl(\frac{1}{P(\Omega)}\displaystyle\int_{\partial\Omega\setminus A_{\delta(\Omega)^{1/4}}}|\H|^{p}\dhn\Bigr)^\frac{1}{p}$.
\end{remark}
\begin{remark}
If we assume that $\partial\Omega$ is connected, then Theorem \ref{Hausdorffcourbure} implies that $\partial\Omega$ is Hausdorff close to a sphere. If $\partial\Omega$ has N connected component, the it asserts that $\partial\Omega$ is Hausdorff close to a sphere union a finite set with at most $N-1$ points.
\end{remark}
Note that in the case $p<n$ we can not control the cardinal of $I$ in terms of $\|\H\|_p$. Indeed, consider the sequence of domains $\Omega_{k}$ obtained by the union of $\B^{n+1}$ and $k$ balls $B_{x_i}(r_i/k)$ where $x_i$ are some points satisfying for instance $\dist(0,x_i)=2i$. If $\displaystyle\sum_{i\geqslant 0} r_i^{n-p}$ is convergent then $\displaystyle\lim_{k\longrightarrow\infty}\delta(\Omega_k)=0$ and $P(\Omega_k)\|\H_k\|_p$ (where $\H_k$ denotes the mean curvature of $\partial\Omega_k$) remains bounded when $\Card(I)$ tends to infinity.

Here also we refer to Theorem \ref{Hausdorffcourbureball} of Section \ref{postponed} for a version of Theorem \ref{Hausdorffcourbure} generalizing inequality \ref{fuglede}.

\subsection{Bound on $\|\H\|_p$ with $p>n$}

When $p>n$, it follows from \ref{cardinal}  that if $\delta(\Omega)$ is small enough then $I=\emptyset$ and $\partial\Omega$ is Hausdorff close to $S_{x_\Omega}(R_{\Omega})$. More precisely we have that 
  
\begin{theorem}\label{Lipschitz} Let $p>n$. There exists a constant $C(n,p)>0$ such that if $\Omega$ is an open set with smooth boundary $\partial\Omega$ such that $\hn(\partial\Omega)\|\H\|_p^n\leqslant K$ and $\delta(\Omega)\leqslant \frac{1}{C(n,p,K)}$ then $\partial\Omega$ is diffeomorphic and quasi-isometric to $S_{x_\Omega}(R_\Omega)$. Moreover the Lipschitz distance $d_L$ satisfies
$$d_L(\partial\Omega,\S_{x_\Omega}(R_{\Omega}))\leqslant C(n,p)\delta(\Omega)^\frac{2(p-n)}{p(n+2)-2n}$$
for any $n\geqslant 2$ and the Hausdorff distance
$$d_H(\partial\Omega,S_{x_\Omega}(R_\Omega))\leqslant C(n,p,K)R_\Omega\delta(\Omega)^\frac{2p-n}{2p-2n+np}$$
when $n\geqslant 3$ and
$$d_H(\partial\Omega,S_{x_\Omega}(R_\Omega))\leqslant C(p,K)R_\Omega(-\delta(\Omega)\ln\delta(\Omega))^\frac{1}{2}$$
when $n=2$.
\end{theorem}

\begin{remark}
Actually, under the assumption of the previous theorem, we show that $\partial\Omega=\{\varphi(w)w,w\in S_{x_\Omega}(R_\Omega)\}$, where $\varphi\in W^{1,\infty}(S_{x_\Omega}(R_\Omega))\cap W^{2,p}(S_{x_\Omega}(R_\Omega))$, with $\|d\varphi\|_\infty\leqslant \frac{C(n,p,K)}{R_\Omega}\delta(\Omega)^\frac{p-n}{2p-2n+pn}$ and $\|\nabla d\varphi\|_p\leqslant C(n,p,K)/R_\Omega^2$. So $\Omega$ is a nearly spherical domain in the sense of Fuglede and is the graph over $S_{x_\Omega}(R_\Omega)$ of a $C^{1,1-\frac{n}{p}}(S_{x_\Omega}(R_\Omega))$ function. It implies that any sequence of domain $(\Omega_k)_k$ with $\delta(\Omega_k)\to0$ and $\hn(\partial\Omega_k)\|\H_k\|_p^n\leqslant K$ converges to $S_{x_\Omega}(R_\Omega)$ in $C^{1,q}$ topology for any $q<1-\frac{p}{n}$.
\end{remark}

\begin{remark}
The estimates on $d_L$ and $d_H$ in Theorem \ref{Lipschitz} are sharp with respect of the exponent of $\delta(\Omega)$ involved, but not for what concern the constant $C(n,p,K)$. We show it by constructing example at the end of section \ref{Allard}. Note moreover that in the case $p=\infty$ we recover the same exponent as in the convex case.
\end{remark}


\subsection{Stability of the Chavel Inequality}
In the last part of this paper we answer a question asked by Bruno Colbois concerning the almost extremal hypersurfaces for the Chavel's inequality: if we set $\lambda_1^{\Sigma}$ the first nonzero eigenvalue  of a compact hypersurface $\Sigma$ that bounds a domain $\Omega$, Chavel's inequality says that 
\begin{equation}\label{Chavel}
\lambda_1^{\Sigma}\leqslant\frac{n}{(n+1)^2}\Bigl(\frac{\hn(\Sigma)}{|\Omega|}\Bigr)^2
\end{equation}
Moreover equality holds if and only if $\Sigma$ is a geodesic sphere. Now if we denote by $\gamma(\Omega)$ the deficit of Chavel's inequality (i.e. $\gamma(\Omega)=\frac{n}{\lambda_1^{\Sigma}(n+1)^2}\Bigl(\frac{\hn(\Sigma)}{|\Omega|}\Bigr)^2-1$), we have 
\begin{theorem}\label{ChavelPinc} Let $\Sigma$ be an embedded compact hypersurface bounding a domain $\Omega$ in $\R^{n+1}$. If $\gamma(\Omega)\leqslant \frac{1}{C(n)}$ then  we have
$$\delta(\Omega)\leqslant C(n)\gamma(\Omega)^{1/2}$$
\end{theorem}
Consequently, $\delta(\Omega)$ can be replaced by $\gamma(\Omega)^\frac{1}{2}$ in all the previous theorems, which gives the stability of the Chavel's inequality. Note moreover that $\gamma$ small implies readily that $\Sigma=\partial \Omega$ is connected and so we have $N=0$ and $I=\emptyset$ is this case.

\section{Preliminaries}\label{sec:Prel}
\subsection{Definitions}
First let us introduce some notations and recall some definitions used in the paper. Throughout the paper we adopt the notation that $C(n,k,p,\cdots)$ is function which depends on $p$, $q$, $n$, $\cdots$. It eases the exposition to disregard the explicit
nature of these functions. The convenience of this notation is that even though $C$ might change from line to line in a calculation it still maintains these basic features.

Given two bounded sets $A$ and $B$ the Hausdorff distance between $A$ and $B$ is defined by
$$d_H(A,B)=\inf\{\varepsilon\mid A\subset B_{\varepsilon}\ \text{and}\ B\subset A_{\varepsilon}\}$$
where for any subset $E$, $E_{\varepsilon}=\{x\in\R^{n+1}\mid \dist(x,E)\leqslant\varepsilon\}$.

Let $\mu$ be a $\R^{n+1}$-valued Borel measure on $\R^{n+1}$. Its total variation is the nonnegative measure $|\mu|$ defined on  any Borel set $\Omega$ by 
$$|\mu|(\Omega):=\sup\left\{\sum_{k\in\N}\|\mu(\Omega_k)\| \ \mid \ \Omega_i\cap \Omega_j=\emptyset\ , \ \bigcup_{k\in\N}\Omega_k\subset \Omega\right\}$$
Given a Borel set $\Omega$ of $\R^{n+1}$, we say that $\Omega$ is of finite perimeter if  the distributional gradient $D\chi_\Omega$ of its characteristic function is a $\R^{n+1}$-valued Borel measure such that $|D\chi_\Omega|(\R^{n+1})<\infty$. The perimeter  of $\Omega$ is then  $P(\Omega):=|D\chi_\Omega|(\R^{n+1})$. Of course if $\Omega$ is a bounded domain with a smooth boundary we have $P(\Omega)=\hn(\partial\Omega)$. For any set $\Omega$ with finite perimeter, we have $P(\Omega)=\hn\bigl(\front(\Omega)\bigr)$ where $\front(\Omega)$ is the reduced boundary defined by
$$\front(\Omega) :=\left\{x\in\R^{n+1}\ \mid\  \forall r>0\ ,\ |D\chi_\Omega|(B_x(r))>0\ \text{and} \ \lim_{r\longrightarrow 0^+}\frac{D\chi_\Omega(B_x(r))}{|D\chi_\Omega|(B_x(r))}\in\S^n\right\}$$
Moreover Federer (see \cite{AmbFusPal}) proved that $\front(\Omega)\subset\partial^{\star}\Omega$ where $\partial^{\star}\Omega$ is the essential boundary of $\Omega$ defined by
$$\partial^{\star}\Omega:=\R^{n+1}\setminus(\Omega^0\cup \Omega^1)$$
where $\Omega^t:=\displaystyle\left\{x\in\R^{n+1}\ \mid\ \lim_{r\longrightarrow 0}\frac{|\Omega\cap B_x(r)|}{|B_x(r)|}=t\right\}$.

\subsection{Some results proved in \cite{FigMagPra}}
Now we gather some results proved in \cite{FigMagPra} about almost isoperimetric sets, that will be used in this paper.
\begin{theorem}\label{inegaliteG}({\sc A. Figalli, F. Maggi, A. Pratelli}, \cite{FigMagPra}) Let $\Omega$ be a set of $\R^{n+1}$ of finite perimeter, with $0<|\Omega|<\infty$ and $\delta(\Omega)\leqslant\min\left(1,\displaystyle\frac{k(n)^2}{8}\right)$ where $k(n):=\displaystyle\frac{2-2^{\frac{n}{n+1}}}{3}$. 
Then there exists a domain $G\subset\Omega$ such that
\begin{enumerate}
\item $0\leqslant|\Omega|-|G|\leqslant|\Omega\setminus G|\leqslant\frac{\delta(\Omega)}{k(n)}|\Omega|$,
\item $P(G)\leqslant P(\Omega)$,
\item $\delta(G)\leqslant\frac{3}{k(n)}\delta(\Omega)$,
\item There exists a point $x_\Omega\in\R^{n+1}$ such that $$\displaystyle\int_{\front(G)}\bigl||x-x_\Omega|-R_{G}\bigr|\dhn\leqslant\frac{10(n+1)^3}{k(n)}|G|\delta(\Omega)^{1/2}$$
where $X$ is the vector position of $\R^{n+1}$,
\item $|G\Delta B_{x_\Omega}(R_G)|\leqslant \frac{20(n+1)^3}{k(n)}|G|\delta(\Omega)^{1/2}$.
\end{enumerate}
\end{theorem}
The following property is important for our purpose and derive easily from \cite{FigMagPra}, but since it is not proved nor stated in \cite{FigMagPra}, we give a proof of it for sake of completeness.
\begin{lemma}\label{comparomegaG} There exists a constant $C(n)>0$ such that under the assumptions and notations of the previous theorem, we have
$$\bigl(1-C(n)\delta(\Omega)\bigr)\hn\bigl(\front(\Omega)\bigr)\leqslant\hn\bigl(\front(\Omega)\cap\front(G)\bigr)$$
\end{lemma}
\begin{proof} We reuse the notations of \cite{FigMagPra}. First of all, by the previous theorem, we have
\begin{align}
\hn\bigl(\front(G)\bigr)&\geqslant I(\B^{n+1})|G|^\frac{n}{n+1}\geqslant (1-\frac{\delta(\Omega)}{k(n)})^\frac{n}{n+1} I(\B^{n+1})|\Omega|^\frac{n}{n+1}\geqslant\frac{(1-\frac{\delta(\Omega)}{k(n)})^\frac{n}{n+1}}{1+\delta(\Omega)}\hn\bigl(\front(\Omega)\bigr)\nonumber\\
&\geqslant\bigl(1-C(n)\delta(\Omega)\bigr)\hn\bigl(\front(\Omega)\bigr)\label{ineq1}
\end{align}
and by the construction made in \cite{FigMagPra}, $\Omega$ is the disjoint union of $G$ and a set $F_\infty$ which satisfy 
$$\hn\bigl(\front(F_\infty)\bigr)\leqslant\bigl(1+ k(n)\bigr)\hn\bigl(\front(\Omega)\cap\front(F_\infty)\bigr).$$
Then we have
$$\hn(\front(\Omega))=\hn(\front(\Omega)\cap\front(G))
+\hn(\front(\Omega)\cap\front(F_{\infty}))$$
and
\begin{align*}
(1+k(n))\hn\bigl(\front(\Omega)\bigr)&+(1-k(n))\hn\bigl(\front(G)\cap\front(\Omega)\bigr)\\
&=
2\hn\bigl(\front(G)\cap\front(\Omega)\bigr)+(1+k(n))\hn\bigl(\front(\Omega)\cap\front(F_\infty)\bigr)\\
&\geqslant2\hn\bigl(\front(G)\cap\front(\Omega)\bigr)+\hn\bigl(\front(F_\infty)\bigr)\\
&= \hn\bigl(\front(G)\bigr)+\hn\bigl(\front(\Omega)\bigr)\\
&\geqslant \bigl(2-C(n)\delta(\Omega)\bigr)\hn\bigl(\front(\Omega)\bigr)
\end{align*}
where we have used Inequality \eqref{ineq1}. We infer that
$$\hn\bigl(\front(G)\cap\front(\Omega)\bigr)\geqslant \bigl(1-C(n)\delta(\Omega)\bigr)\hn\bigl(\front(\Omega)\bigr).$$
\end{proof}
\subsection{Proof of remark \ref{DomainHausdorff}}
Up to a translation we can assume that $x_\Omega=0$ and from the Theorem \ref{inegaliteG} we have :
$$|G\Delta B_0(R_G)|\leqslant C(n)|G|\delta(\Omega)^{1/2}$$
Since
$$\Omega\Delta B_0(R_{\Omega})\subset(\Omega\Delta G)\cup(G\Delta B_0(R_G))\cup(B_0(R_G)\Delta B_0(R_{\Omega}))$$
we deduce immediately that $|\Omega\Delta B_0(R_{\Omega})|\leqslant C(n)|\Omega|\delta(\Omega)^{1/2}$ which proves the point (1) of the remark.

On the other hand let $x\in B_0(R_{\Omega})$ and $R_\Omega\geqslant\varepsilon>0$ such that 
$$B_x(\varepsilon)\cap(\Omega\cap B_0(R_{\Omega}))=\emptyset.$$
We then have $B_x(\varepsilon)\cap B_0(R_{\Omega})\subset\Omega\Delta B_0(R_{\Omega})$ and since 
$B_x(\varepsilon)\cap B_0(R_{\Omega})$ contains the ball with diameter $\R x\cap B_x(\varepsilon)\cap B_0(R_{\Omega})$ whose length is larger than $\varepsilon$, we get
$$\frac{1}{C(n)}\varepsilon^{n+1}\leqslant|\Omega\Delta B_0(R_{\Omega})|\leqslant C(n)|\Omega|\delta(\Omega)^{1/2}$$
Since $\Omega\cap B_0(R_\Omega)\subset B_0(R_\Omega)$, it suffices to get the point (2) that is 
$$d_H(\Omega\cap B_0(R_{\Omega}),B_0(R_{\Omega}))\leqslant C(n) |\Omega|^{\frac{1}{n+1}}\delta(\Omega)^{\frac{1}{2(n+1)}}$$
\section{Concentration in a tubular neighborhood of a sphere}
The main result of this section is the following theorem :
\begin{lemma}\label{concentration} Let $\Omega$ be a set of $\R^{n+1}$ with finite perimeter and let 
$$A_{\eta}:=\Bigl\{x\in\R^{n+1}/\, \bigl||x-x_\Omega|-R_{\Omega}\bigr|\leqslant R_{\Omega}\eta\Bigr\}.$$
 If $\delta(\Omega)\leqslant \frac{1}{C(n)}$ then for any $\alpha\in(0,\frac{1}{2})$, we have 
$$\hn\bigl(\front(\Omega)\setminus A_{\delta(\Omega)^{\alpha}}\bigr)\leqslant C(n)P(\Omega)\delta(\Omega)^{\frac{1}{2}-\alpha}$$
\end{lemma}

\begin{proof} By inequalities (4) and (1) of Theorem \ref{inegaliteG}, we get
\begin{eqnarray*}
\hn\bigl(\front(G)\setminus A_{\eta}\bigr)&\leqslant&\frac{1}{R_\Omega \eta}\int_{\front(G)\setminus A_{\eta}}\bigl||x-x_\Omega|-R_\Omega\bigr|\dhn\\
&\leqslant&\frac{1}{R_\Omega \eta}\int_{\front(G)\setminus A_{\eta}}\bigl||x-x_\Omega|-R_G\bigr|\dhn+\frac{|R_\Omega-R_G|}{\eta R_\Omega}\hn\bigl(\front(G)\bigr)\\
&\leqslant& \frac{1}{R_\Omega\eta}\frac{10(n+1)^3}{k(n)}|G|\delta(\Omega)^{1/2}+\frac{|R_\Omega-R_G|}{\eta R_\Omega}\hn\bigl(\front(G)\bigr)\\
&\leqslant& C(n)\frac{|\Omega|^\frac{n}{n+1}}{\eta}\sqrt{\delta(\Omega)}
\end{eqnarray*}
where we have used that $\hn\bigl(\front(G)\bigr)=P(G)=C(n)\bigl(1+\delta(G)\bigr) |G|^\frac{n}{n+1}\leqslant C(n)|\Omega|^\frac{n}{n+1}$ (by Theorem \ref{inegaliteG} (1) and (3)).

Now by Lemma \ref{comparomegaG} and Inequality (2) of Theorem \ref{inegaliteG}, we have
$$\hn\bigl(\front(G)\setminus \bigl(\front(\Omega)\cap\front(G)\bigr)\bigr)\leqslant\hn\bigl(\front(\Omega)\setminus \bigl(\front(\Omega)\cap\front(G)\bigr)\bigr)\leqslant C(n)\delta(\Omega)\hn\bigl(\front(\Omega)\bigr)$$
And so
\begin{align*}\hn\bigl(\front(\Omega)\setminus A_{\eta}\bigr)&\leqslant\hn\bigl(\bigl(\front(\Omega)\cap\front(G)\bigr)\setminus A_{\eta}\bigr)+\hn\bigl(\front(\Omega)\setminus \bigl(\front(\Omega)\cap\front(G)\bigr)\bigr)\\
&\leqslant\hn\bigl(\front(G)\setminus A_{\eta}\bigr)+C(n)\hn\bigl(\front(\Omega)\bigr)\delta(\Omega)\\
&\leqslant\frac{C(n)}{\eta}|\Omega|^\frac{n}{n+1}\delta(\Omega)^{1/2}+C(n)|\Omega|^{\frac{n}{n+1}}\delta(\Omega)\\
&\leqslant C(n)\left(\frac{1}{\eta}\delta(\Omega)^{1/2}+\delta(\Omega)\right)|\Omega|^{\frac{n}{n+1}}
\end{align*}

Then choosing $\eta:=\delta(\Omega)^{\alpha}$ and $\delta(\Omega)\leqslant 1$ we get the desired result.
\end{proof}
\section{Domains with small deficit without assumption on the boundary}
In this section, we gather the proofs of several geometric-measure properties of the boundary of almost isoperimetric domains.

\subsection{Proof of Theorem \ref{Hausdorff}}
 By Lemma \ref{concentration}, we have Inequality (1) with $A(\Omega)=\front(\Omega)\cap A_{\delta(\Omega)^{1/4}}$.
Inequality (2) will be a consequence of the following density theorem.
\begin{theorem}\label{equidense}  Let $\Omega$ be a set of $\R^{n+1}$ with finite perimeter and  $\rho\in\bigl[C(n)\delta(\Omega)^\frac{1}{8}R_\Omega,R_\Omega\bigr]$. Then for any $x\in S_{x_\Omega}(R_{\Omega})$ we have
$$\Bigl|\frac{\hn\bigl(B_x(\rho)\cap S_{x_\Omega}(R_\Omega)\bigr)}{R_\Omega^n\Vol\S^n}-\frac{\hn\bigl(\front(\Omega)\cap B_x(\rho)\bigr)}{\hn\bigl(\front(\Omega)\bigr)}\Bigr|\leqslant C(n)\delta(\Omega)^\frac{1}{4}$$
\end{theorem}
Let $x\in S_{x_\Omega}(R_{\Omega})$ and $\rho=C_1(n)R_\Omega\delta(\Omega)^{\beta(n)}$ with $\beta(n):=\min(\frac{1}{8},\frac{1}{4n})$. Then for $C_1(n)$ large enough and $\delta(\Omega)\leqslant(1/C_1(n))^{1/\beta(n)}$, $\rho\in\bigl[C(n)\delta(\Omega)^\frac{1}{8}R_\Omega,R_\Omega\bigr]$ and the estimate of Theorem \ref{equidense}  combined to the fact that there exists a constant $C_2(n)$ such that
$$\frac{\hn\bigl(B_x(\rho)\cap S_{x_\Omega}(R_\Omega)\bigr)}{R_\Omega^n\Vol\S^n}\geqslant C_2(n)\left(\frac{\rho}{R_\Omega}\right)^n$$
gives for $C_1(n)$ great enough
\begin{align*}\frac{\hn\bigl(\front(\Omega)\cap B_x(\rho)\bigr)}{\hn\bigl(\front(\Omega)\bigr)}&\geqslant-C(n)\delta(\Omega)^{1/4}+C_2(n)\left(\frac{\rho}{R_\Omega}\right)^n\\
&\geqslant(-C(n)+C_2(n)C_1(n)^n)\delta(\Omega)^{\min(\frac{n}{8},\frac{1}{4})}\\
&\geqslant C_3(n)\delta(\Omega)^{\min(\frac{n}{8},\frac{1}{4})}
\end{align*}
Moreover from the lemma \ref{concentration} we have 
\begin{align*}\hn(\front(\Omega)\cap A_{\delta(\Omega)^{1/4}}\cap B_x(\rho))&\geqslant\hn\bigl(\front(\Omega)\cap B_x(\rho)\bigr)-C(n)P(\Omega)\delta(\Omega)^{1/4}\\
&\geqslant C_3(n)P(\Omega)\delta(\Omega)^{\min(\frac{n}{8},\frac{1}{4})}-C(n)P(\Omega)\delta(\Omega)^{1/4}\\
&\geqslant (C_3(n)-C(n))P(\Omega)\delta(\Omega)^{\min(\frac{n}{8},\frac{1}{4})}\\
&\geqslant C_4(n)P(\Omega)\delta(\Omega)^{\min(\frac{n}{8},\frac{1}{4})}
\end{align*}
If $C_1(n)$ is large enough. This implies that $\front(\Omega)\cap A_{\delta(\Omega)^{1/4}}\cap B_x(\rho)$ has non-zero measure, hence is non-empty for any $x\in S_{x_\Omega}(R_\Omega)$. Putting $A(\Omega)=\front(\Omega)\cap A_{\delta(\Omega)^{1/4}}$, we obtain that $d_H(A(\Omega),S_{x_\Omega}(R_\Omega))\leqslant\rho$ for $C_1(n)$ large enough which gives the fact (2) of Theorem \ref{Hausdorff}. \hfill$\qedbox$
\medskip

Note that Theorem \ref{equidense} implies that density of $\front(\Omega)$ near each point of $\S^n(R_{\Omega})$ converges to 1 at any fixed scale. It will be combined with Allard's regularity theorem in Section \ref{Allard} to prove Theorem \ref{Lipschitz}.

\subsection{Proof of Theorem \ref{equidense}}

 It will be a consequence of the following fundamental proposition.
\begin{proposition}\label{frontOmegaSnROmega} Let $\Omega$ be a set of $\R^{n+1}$ of finite perimeter, with $\delta(\Omega)\leqslant\frac{1}{C(n)}$. For any $f\in C^1_c(\R^{n+1})$, we have
$$\Bigl|\frac{1}{P(\Omega)}\int_{\front(\Omega)}f\dhn
-\frac{1}{R_\Omega^n\Vol\S^n}\int_{S_{x_\Omega}(R_{\Omega})}f\dhn\Bigr|\leqslant C(n)\bigl(\|f\|_\infty+\|df\|_\infty\bigr)\delta(\Omega)^\frac{1}{2},$$
where we denote $\|df\|_\infty=\displaystyle\sup_y|d_{y} f(y)|$.
\end{proposition}
\begin{proof}
Up to translation, we can assume that $x_\Omega=0$ subsequently.
Let $G\subset\Omega$ be the subset associated to $\Omega$ in Theorem \ref{inegaliteG}. We note $X$ the field $X_x=x$ for any $x\in\Omega$. We have $\div_x(fX)=df_x(X_x)+(n+1)f(x)$ and so we get
\begin{align}
\Bigl|\int_{\front(G)}f\langle X,\nu_G\rangle\dhn&-R_G\int_{S_0(R_G)}f\dhn\Bigr|\nonumber\\
=&\Bigl|\int_{\front(G)}f\langle X,\nu_G\rangle\dhn-\int_{S_0(R_G)}f\langle X,\nu_{S_{0}(R_G)}\rangle\dhn\Bigr|\nonumber\\
=&\Bigl|\int_G\div(fX)\dhnpi-\int_{B_0(R_G)}\div(fX)\dhnpi\Bigr|\nonumber\\
\leqslant&(n+1)\Bigl|\int_Gf\dhnpi-\int_{B_0(R_G)}f\dhnpi\Bigr|\nonumber\\
&+\Bigl|\int_G df(X)\dhnpi-\int_{B_0(R_G)}df(X)\dhnpi\Bigr|\nonumber\\
\leqslant&(n+1)\bigl(\|f\|_\infty+\|df\|_\infty\bigr)|G\Delta B_0(R_G)|\nonumber\\
\leqslant&\frac{20(n+1)^4}{k(n)}|G|\bigl(\|f\|_\infty+\|df\|_\infty\bigr)\delta(\Omega)^{1/2}\label{ineq3}
\end{align}
Where we have used Inequality (5) of Theorem \ref{inegaliteG}.
Now we have
\begin{align*}&\int_{\front(G)}\bigl|f(R_G-\langle X,\nu_G\rangle)\bigr|\dhn\\
&\leqslant\|f\|_\infty\int_{\front(G)}\bigl|R_G-|X|\bigr|\dhn+\|f\|_\infty\int_{\front(G)}\bigl||X|-\langle X,\nu_G\rangle\bigr|\dhn\\
&=\|f\|_\infty\int_{\front(G)}\bigl|R_G-|X|\bigr|\dhn+\|f\|_\infty\int_{\front(G)}|X|\dhn- \|f\|_\infty\int_{\front(G)}\langle X,\nu_G\rangle\dhn\\
&\leqslant 2\|f\|_\infty\int_{\front(G)}\bigl|R_G-|X|\bigr|\dhn+R_G\|f\|_\infty\hn\bigl(\front(G)\bigr)-\|f\|_\infty\int_G\div(X)\dhnpi\\
&\leqslant\frac{20(n+1)^3}{k(n)}|G|\|f\|_\infty\delta(\Omega)^{1/2}+\|f\|_\infty R_G\hn\bigl(\front(G)\bigr)-\|f\|_\infty(n+1)|G|
\end{align*}
Now a straightforward computation shows that $\delta(G)=\frac{R_G\hn\bigl(\front(G)\bigr)}{(n+1)|G|}-1$. Consequently
\begin{equation}\label{ineq4}
\int_{\front(G)}\bigl|f(R_G-\langle X,\nu_G\rangle)\bigr|\dhn\leqslant\|f\|_\infty C(n)|G|\delta(\Omega)^{1/2}
\end{equation}
Combining Inequalities \eqref{ineq3} and \eqref{ineq4} gives
\begin{eqnarray}
\frac{1}{P(\Omega)}\Bigl|\int_{\front(G)}f\dhn-\int_{\S^n(R_G)}f\dhn\Bigr|&\leqslant& C(n)\frac{|G|}{P(\Omega)R_G}\bigl(\|f\|_\infty+\|df\|_\infty\bigr)\delta(\Omega)^\frac{1}{2}\nonumber\\
&\leqslant&C(n)\bigl(\|f\|_\infty+\|df\|_\infty\bigr)\delta(\Omega)^\frac{1}{2}\label{ineq5}
\end{eqnarray}
Where we have used  Inequality (1) of Theorem \ref{inegaliteG} to get 
\begin{equation}\label{ineq14}
\frac{|G|}{P(\Omega) R_G}\leqslant C(n)\frac{|G|^\frac{n}{n+1}}{P(\Omega)}\leqslant C(n)\frac{|\Omega|^\frac{n}{n+1}}{P(\Omega)}\leqslant C(n).
\end{equation}

We have
$$\left|\frac{1}{P(\Omega)}\int_{\front(\Omega)}f\dhn
-\frac{1}{R_\Omega^n\Vol\S^n}\int_{\S^n(R_{\Omega})}f\dhn\right|\leqslant A_1+A_2+A_3+A_4$$
with
\begin{align*}
A_1&=\frac{1}{P(\Omega)}\left|\int_{\front(\Omega)}f\dhn-\int_{\front(G)}f\dhn\right|\\
A_2&=\frac{1}{P(\Omega)}\left|\int_{\front(G)}f\dhn-\int_{\S^n(R_G)}f\dhn\right|\\
A_3&=\frac{1}{P(\Omega)}\left|\int_{\S^n(R_G)}f\dhn-\int_{\S^n(R_{\Omega})}f\dhn\right|\\
A_4&=\bigl|\frac{R_\Omega^n\Vol\S^n}{P(\Omega)}-1\bigr|\|f\|_\infty\leqslant C(n)\delta(\Omega)\|f\|_\infty
\end{align*}

Note that $A_2$ is controlled by Inequality \eqref{ineq5}. Let us now estimate $A_1$. 
By Lemma \ref{comparomegaG} we have
\begin{align*} A_1&=\frac{1}{P(\Omega)}\Bigl|\int_{\front(\Omega)\setminus(\front(\Omega)\cap \front(G))}f\dhn-\int_{\front(G)\setminus(\front(G)\cap
\front(\Omega))}f\dhn\Bigr|\\
&\leqslant\|f\|_{\infty} C(n)\frac{|\Omega|^{\frac{n}{n+1}}}{P(\Omega)}\delta(\Omega)\leqslant\|f\|_{\infty} C(n)\delta(\Omega)\\
A_3&\leqslant\frac{1}{P(\Omega)}\int_{\S^n}\bigl|R_G^nf(R_Gu)-R_\Omega^nf(R_\Omega u)\bigr|\dhn\\
&\leqslant\frac{1}{P(\Omega)}\int_{\S^n}\bigl|(R_G^n-R_\Omega^n)f(R_Gu)\bigr|+R_\Omega^n\bigl|f(R_\Omega u)-f(R_Gu)\bigr|\dhn\\
&\leqslant C(n)\frac{|\Omega|^\frac{n}{n+1}-|G|^\frac{n}{n+1}}{P(\Omega)}\|f\|_{\infty}+\frac{R_\Omega^n}{P(\Omega) R_G}\int_{\S^n}\int_{R_G}^{R_\Omega}|d_{tu}f(tu)|\,dt\,du\\
&\leqslant C(n)(\|f\|_{\infty}+\|df\|_{\infty})\delta(\Omega)
\end{align*}
where once again we have used the estimates of Theorem \ref{inegaliteG}.
\end{proof}

\begin{proofequidense} Up to translation, we can assume that $x_\Omega=0$. Let $\rho\leqslant R_\Omega$. By Lemma \ref{concentration}, we have
\begin{equation}\label{ineq11}
\Bigl|\frac{\hn\bigl(\front(\Omega)\cap B_x(\rho)\bigr)}{\hn\bigl(\front(\Omega)\bigr)}-\frac{\hn\bigl(\front(\Omega)\cap B_x(\rho)\cap A_{\delta(\Omega)^\frac{1}{4}}\bigr)}{\hn\bigl(\front(\Omega)\bigr)}\Bigr|\leqslant C(n)\delta(\Omega)^{\frac{1}{4}}
\end{equation}
We set $\eta=\delta(\Omega)^\frac{1}{4}\leqslant\frac{1}{2}$ and $\varphi:[0,+\infty)\to[0,1]$ be a $C^1$ function with compact support in $(0,2R_\Omega)$, $\frac{2}{R_\Omega}$-Lipschitz and such that $\varphi(t)=1$ on $\bigl[R_\Omega(1-\eta),R_\Omega(1+\eta)\bigr]$. For any function $v\in C^1\bigl(S_0(R_\Omega)\bigr)$, we set $f(x)=\varphi(|x|)v(\frac{R_\Omega x}{|x|})$. Then $|df_x(x)|\leqslant 4\|v\|_{\infty}$ and applying Proposition \ref{frontOmegaSnROmega} to $f$, we get
\begin{equation}\label{ineq6}
\Bigl|\frac{1}{P(\Omega)}\int_{\front(\Omega)}f\dhn
-\frac{1}{R_\Omega^n\Vol\S^n}\int_{S_0(R_{\Omega})}v\dhn\Bigr|\leqslant C(n)\|v\|_\infty\delta(\Omega)^\frac{1}{2}
\end{equation}
Let $x\in S_0(R_{\Omega})$ and $v_r$ be the characteristic function of the geodesic ball of center $x$ and radius $r$ in $S_0(R_{\Omega})$. By convolution, we can approximate $v_r$ in $L^1(S_0(R_\Omega))$ by $C^1$ functions $u_k$ such that $\|u_k\|_\infty\leqslant 1$. Applying Inequality \eqref{ineq6} to $v=u_k$ and letting $k$ tends to $\infty$, we get 
\begin{equation}\label{ineq7}
\Bigl|\frac{1}{P(\Omega)}\int_{\front(\Omega)}f_r\dhn
-\frac{\hn(\mathcal{C}_{r/R_\Omega}\cap S_0(R_\Omega))}{R_\Omega^n\Vol\S^n}\Bigr|\leqslant C(n)\delta(\Omega)^\frac{1}{2}
\end{equation}
where $f_r=\varphi_{R_\Omega}(\|x\|)v_r(\frac{R_\Omega x}{\|x\|})$ and  where $\mathcal{C}_{\alpha}=\bigl\{y\in\R^{n+1}\setminus\{0\}/\,\langle\frac{y}{\|y\|},\frac{x}{\|x\|}\rangle\geqslant \cos\alpha\bigr\}$.
Now, since $\|f_r\|_\infty\leqslant1$, Lemma \ref{concentration} gives us
\begin{equation}\label{ineq8}
\frac{1}{P(\Omega)}\Bigl|\int_{\front(\Omega)}f_r\dhn-\int_{\front(\Omega)\cap A_{\delta(\Omega)^\frac{1}{4}}}f_r\dhn\Bigr|\leqslant C(n)\delta(\Omega)^{\frac{1}{4}}
\end{equation}
By construction of $f_r$, we have
\begin{equation}\label{ineq9}
\int_{\front(\Omega)\cap A_{\delta(\Omega)^\frac{1}{4}}}f_r\dhn=\hn\bigl(\front(\Omega)\cap A_{\delta(\Omega)^\frac{1}{4}}\cap\mathcal{C}_{r/R_{\Omega}}\bigr)
\end{equation}
Combining Inequalities \eqref{ineq7}, \eqref{ineq8} and \eqref{ineq9}, we get
\begin{equation}\label{ineq10}
\Bigl|\frac{\hn\bigl(\front(\Omega)\cap A_{\delta(\Omega)^\frac{1}{4}}\cap\mathcal{C}_{r/R_{\Omega}}\bigr)}{\hn\bigl(\front(\Omega)\bigr)}
-\frac{\hn\bigl(\mathcal{C}_{r/R_{\Omega}}\cap \S_0(R_\Omega)\bigr)}{R_\Omega^n\Vol\S^n}\Bigr|\leqslant C(n)\delta(\Omega)^{\frac{1}{4}}
\end{equation}

We now assume that $\delta(\Omega)^\frac{1}{4}\leqslant\frac{\rho^2}{2R_\Omega^2}$. The following angles
$${\alpha}_{\mathrm ext}=\arccos\bigl(\frac{1+(1-\delta(\Omega)^\frac{1}{4})^2-\frac{\rho^2}{R_\Omega^2}}{2(1-\delta(\Omega)^\frac{1}{4})}\bigr)\ \ \mbox{ and }\ \ {\alpha}_{\mathrm int}=\arccos\bigl(\frac{1+(1+\delta(\Omega)^\frac{1}{4})^2-\frac{\rho^2}{R_\Omega^2}}{2(1+\delta(\Omega)^\frac{1}{4})}\bigr)$$
satisfy the following property (see figure \ref{fig})
\begin{figure}\label{fig}
\centering
\def\svgwidth{0.35\textwidth}
{\tiny
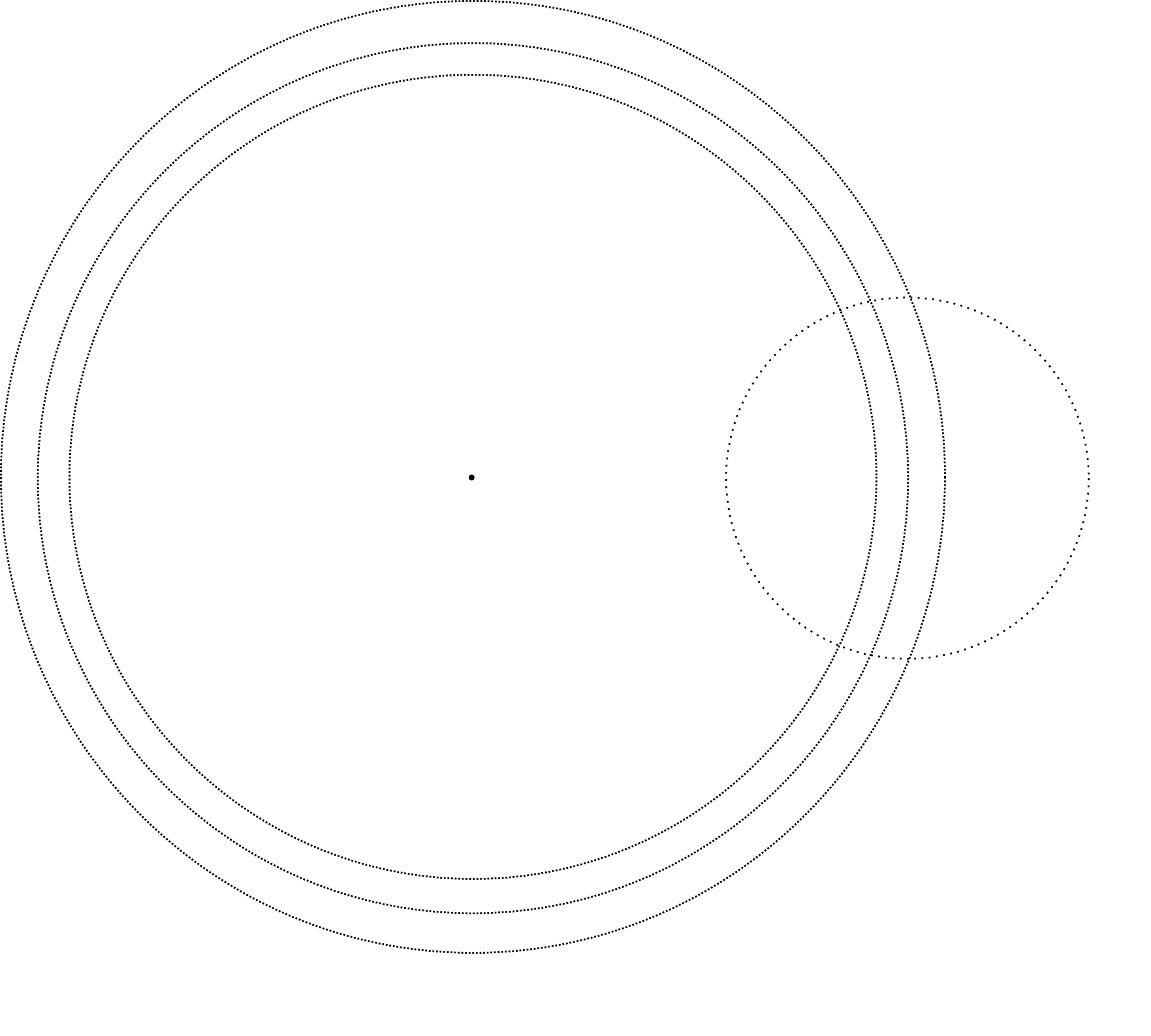}
\caption{}
\end{figure}
$$\mathcal{C}_{\mathrm int}\cap A_{\delta(\Omega)^\frac{1}{4}}\subset B_x(\rho)\cap A_{\delta(\Omega)^\frac{1}{4}}\subset\mathcal{C}_{\mathrm ext}\cap A_{\delta(\Omega)^\frac{1}{4}},$$
where we have set $C_{\mathrm int}=C_{{\alpha}_{\mathrm int}}$ and $C_{\mathrm ext}=C_{{\alpha}_{\mathrm ext}}$, so we get the following inequalities
\begin{align*}
\frac{\hn\bigl(\front(\Omega)\cap B_x(\rho)\cap A_{\delta(\Omega)^\frac{1}{4}}\bigr)}{\hn\bigl(\front(\Omega)\bigr)}&\geqslant \frac{\hn\bigl(\front(\Omega)\cap \mathcal{C}_{\mathrm int}\cap A_{\delta(\Omega)^\frac{1}{4}}\bigr)}{\hn\bigl(\front(\Omega)\bigr)}\\
&\geqslant\frac{\hn\bigl(\mathcal{C}_{\mathrm int}\cap \S_0(R_\Omega)\bigr)}{R_\Omega^n\Vol\S^n}-C(n)\delta(\Omega)^{\frac{1}{4}}\\
\frac{\hn\bigl(\front(\Omega)\cap B_x(\rho)\cap A_{\delta(\Omega)^\frac{1}{4}}\bigr)}{\hn\bigl(\front(\Omega)\bigr)}&\leqslant \frac{\hn\bigl(\front(\Omega)\cap \mathcal{C}_{\mathrm ext}\cap A_{\delta(\Omega)^\frac{1}{4}}\bigr)}{\hn\bigl(\front(\Omega)\bigr)}\\
&\leqslant\frac{\hn\bigl(\mathcal{C}_{\mathrm ext}\cap \S_0(R_\Omega)\bigr)}{R_\Omega^n\Vol\S^n}+C(n)\delta(\Omega)^{\frac{1}{4}}
\end{align*}
Since we have $B_x(\rho)\cap S_0(R_\Omega)=\mathcal{C}_{{\alpha}_\rho}\cap S_0(R_\Omega)$ for ${\alpha}_\rho=\arccos(1-\frac{\rho^2}{2R_\Omega^2})$, we infer the estimate
$$\displaylines{D=\Bigl|\frac{\hn\bigl(\front(\Omega)\cap B_x(\rho)\cap A_{\delta(\Omega)^\frac{1}{4}}\bigr)}{\hn\bigl(\front(\Omega)\bigr)}-\frac{B_x(\rho)\cap S_0(R_\Omega)}{R_\Omega^n\Vol\S^n}\Bigr|\hfill\cr
\hfill\leqslant \frac{\hn\bigl(\mathcal{C}_{\mathrm ext}\cap S_0(R_\Omega)\bigr)-\hn\bigl(\mathcal{C}_{\mathrm int}\cap S_0(R_\Omega)\bigr)}{R_\Omega^n\Vol\S^n}+C(n)\delta(\Omega)^{\frac{1}{4}}}$$
Now, by the Bishop's and Bishop-Gromov's theorems, we have
$$\hn\bigl(\mathcal{C}_{\mathrm ext}\cap S_0(R_\Omega)\bigr)=\hn(B_x^{S_0(R_{\Omega})}(R_{\Omega}\alpha_{\mathrm ext}))\leqslant R_{\Omega}^n\alpha_{\mathrm ext}^n\Vol(\B^n)=\frac{\Vol(\S^{n-1})}{n}R_{\Omega}^n\alpha_{\mathrm ext}^n$$
and
$$\frac{\hn(B_x^{S_0(R_{\Omega})}(R_{\Omega}\alpha_{\mathrm int}))}{\hn(B_0(R_{\Omega}\alpha_{\mathrm int}))}\geqslant\frac{\hn(B_x^{S_0(R_{\Omega})}(R_{\Omega}\alpha_{\mathrm ext}))}{\hn(B_0(R_{\Omega}\alpha_{\mathrm ext}))}$$
that is
$$\frac{\hn\bigl(\mathcal{C}_{\mathrm int}\cap S_0(R_\Omega)\bigr)}{\hn\bigl(\mathcal{C}_{\mathrm ext}\cap S_0(R_\Omega)\bigr)}\geqslant\frac{\alpha_{\mathrm int}^n}{\alpha_{\mathrm ext}^n},$$
where $B_x^{S_0(R_{\Omega})}(r)$ denotes the ball of center $x$ and radius $r$ in $S_0(R_{\Omega})$. These inequalities give
\begin{align*} D&\leqslant\frac{1}{R_{\Omega}^n\Vol(\S^n)}\left(1-\frac{\alpha_{\mathrm int}^n}{\alpha_{\mathrm ext}^n}\right)\frac{\Vol(\S^{n-1})}{n}R_{\Omega}^n\alpha_{\mathrm ext}^n+C(n)\delta(\Omega)^{\frac{1}{4}}\notag\\
&\leqslant\frac{\Vol(\S^{n-1})}{\Vol(\S^n)}\pi^{n-1}|\alpha_{\mathrm ext}-\alpha_{\mathrm int}|+C(n)\delta(\Omega)^{\frac{1}{4}}
\end{align*}
Since by assumption $\delta(\Omega)^\frac{1}{4}\leqslant\frac{\rho^2}{2R_\Omega^2}$, we get $|{\alpha}_{\mathrm ext}-{\alpha}_{\mathrm int}|\leqslant C(n)\delta(\Omega)^\frac{1}{4}$ which gives
$$D\leqslant C(n)\delta(\Omega)^\frac{1}{4}$$
Finally, by Lemma \ref{concentration}, we have
$$\Bigl|\frac{\hn\bigl(\front(\Omega)\cap B_x(\rho)\cap A_{\delta(\Omega)^\frac{1}{4}}\bigr)}{\hn\bigl(\front(\Omega)\bigr)}-\frac{\hn\bigl(\front(\Omega)\cap B_x(\rho)\bigr)}{\hn\bigl(\front(\Omega)\bigr)}\Bigr|\leqslant C(n)\delta(\Omega)^\frac{1}{4}$$
which gives
$$\Bigl|\frac{\hn\bigl(B_x(\rho)\cap S_0(R_\Omega)\bigr)}{R_\Omega^n\Vol\S^n}-\frac{\hn\bigl(\front(\Omega)\cap B_x(\rho)\bigr)}{\hn\bigl(\front(\Omega)\bigr)}\Bigr|\leqslant C(n)\delta(\Omega)^\frac{1}{4}.$$
\end{proofequidense}

\subsection{A control of the unit normal to $\front(\Omega)$}

In this subsection, we prove a result that we will use latter. It gives a weak control of the oscillation of the tangent planes of $\front(\Omega)$. Note that another proof of this result is proposed in \cite{FusJul}.

\begin{lemma}\label{ZL2}
Let $\Omega$ be a set of finite perimeter such that $\delta(\Omega)\leqslant\frac{1}{C(n)}$. Then we have
$$\int_{\front(\Omega)}\left|\nu_\Omega-\frac{x}{|x|}\right|^2\dhn\leqslant C(n)P(\Omega)\delta(\Omega)^\frac{1}{2}$$
\end{lemma}

\begin{proof}
By Lemma \ref{comparomegaG} and the fact that $\nu_\Omega=\nu_G$ $\hn$-almost everywhere in $\front(G)\cap\front(\Omega)$, we have
\begin{align*}
\Bigl|\int_{\front(\Omega)}&\left|\nu_\Omega-\frac{x}{|x|}\right|^2\dhn-\int_{\front(G)}\left|\nu_G-\frac{x}{|x|}\right|^2\dhn\Bigr|\\
&=\Bigl|\int_{\front(\Omega)\setminus\front(G)}\left|\nu_\Omega-\frac{x}{|x|}\right|^2\dhn-\int_{\front(G)\setminus\front(\Omega)}\left|\nu_G-\frac{x}{|x|}\right|^2\dhn\Bigr|\leqslant 4C(n)\delta(\Omega)P(\Omega)
\end{align*}
Now,  we have
\begin{align*}
\int_{\front(G)}\left|\nu_G-\frac{x}{|x|}\right|^2\dhn=2\int_{\front(G)}\left(1-\left\langle\nu_G,\frac{x}{|x|}\right\rangle\right)\dhn
\end{align*}
and by inequality \eqref{ineq4}, we have that
\begin{align*}
\Bigl|\int_{\front(G)}\left(1-\left\langle\nu_G,\frac{x}{|x|}\right\rangle\right)\dhn\Bigr|&\leqslant\Bigl|\int_{\front(G)}\left(1-\left\langle\nu_G,\frac{x}{R_\Omega}\right\rangle\right)\dhn\Bigr|\\
&+\Bigl|\int_{\front(G)}\left\langle\nu_G,\frac{x}{|x|}-\frac{x}{R_\Omega}\right\rangle\dhn\Bigr|\\
&\leqslant C(n)R_G^n\delta(\Omega)^\frac{1}{2}+\frac{1}{R_\Omega}\int_{\front(G)}\bigl||x|-R_\Omega\bigr|\dhn\\
&\leqslant C(n)R_G^n\delta(\Omega)^\frac{1}{2}\leqslant C(n)P(G)\delta(\Omega)^\frac{1}{2}\leqslant C(n)P(\Omega)\delta(\Omega)^\frac{1}{2}
\end{align*}
where the last inequality comes from fact (4) of Theorem \ref{inegaliteG}.
\end{proof}
\subsection{A stability result involving the Preiss distance}\label{ProximityPreiss}
First we recall the definition of the Preiss distance on Radon measures of $\R^{n+1}$.

\begin{definition}
Let $\mu$ and $\nu$ be two Radon measures on $\R^{n+1}$, for any $i\in\N$, we set
$$F_i(\mu,\nu)=\sup\Bigl\{\bigl|\int fd\mu-\int fd\nu\bigr|,\ {\mathrm spt}\, f\subset B_0(i),\ f\geqslant 0,\ {\mathrm Lip}\,f\leqslant 1\Bigr\}$$
and
$$d_P(\mu,\nu)=\sum_{i\in \N}\frac{1}{2^i}\min\bigl(1,F_i(\nu,\mu)\bigr)$$
it gives a distance on the Radon measure of $\R^{n+1}$ whose converging sequences are the weakly$^\star$ converging sequences.
\end{definition}
For almost isoperimetric domains we have a control on the boundary in term of Preiss distance
\begin{theorem}\label{Preiss} Let $\Omega$ be a set of $\R^{n+1}$ with finite perimeter. Then there exists $x_\Omega\in\R^{n+1}$ such that
\begin{equation}\label{DistPreiss}
d_{P}\left(\frac{|D\chi_{B_{x_\Omega}(R_\Omega)}|}{R_\Omega^n\Vol\S^{n}},\frac{|D\chi_{\Omega}|}{P(\Omega)}\right)\leqslant C(n)\sqrt{\delta(\Omega)} 
\end{equation}
where $d_{P}$ is the Preiss distance on Radon measures of $\R^{n+1}$.
\end{theorem}
\begin{proof} Note that if $f$ has support in $B_0(i)$ and is $1$-Lipschitz, then by convolution, it can be uniformly approximated by a sequence of $1$-Lipschitz, $C^1$ and compactly supported functions $(f_k)$. We then have $\displaystyle\lim_k \|f_k\|_\infty=\|f\|_\infty\leqslant i$ and $\displaystyle\lim_k|d_Xf_k(X)|\leqslant\displaystyle\lim_k\|df_k\|_\infty i\leqslant i$ and applying Proposition \ref{frontOmegaSnROmega} to $f_k$ and letting $k$ tends to $\infty$ gives us
$$\Bigl|\frac{1}{P(\Omega)}\int_{\front(\Omega)}f\dhn
-\frac{1}{R_\Omega^n\Vol\S^n}\int_{S_{x_\Omega}(R_{\Omega})}f\dhn\Bigr|\leqslant 2iC(n)\delta(\Omega)^\frac{1}{2}$$
and so
$$F_i\Bigl(\frac{|D\chi_\Omega|}{P(\Omega)},\frac{|D\chi_{B_{x_\Omega}(R_G)}|}{P(B_{x_\Omega}(R_G))}\Bigr)\leqslant 2iC(n)\sqrt{\delta(\Omega)}$$
Hence we get that if $\delta(\Omega)\leqslant\frac{1}{C(n)}$,  then we have
$$d_P\left(\frac{|D\chi_\Omega|}{P(\Omega)},\frac{|D\chi_{B_{x_\Omega}(R_G)}|}{P(B_{x_\Omega}(R_G))}\right)\leqslant C(n)\sqrt{\delta(\Omega)}$$
Since for any couple of measures $\mu,\nu$ we have $d_P(\mu,\nu)\leqslant 2$, we infer that we can leave the condition $\delta(\Omega)\leqslant\frac{1}{C(n)}$ as soon as we consider a larger $C(n)$.
\end{proof}

\section{Domains with small deficit and $\|\H\|_p$ bounded in the case $p\leqslant n$}
\subsection{Proof of Theorems \ref{distancefrontspheremodel} and \ref{Hausdorffcourbure}}
These theorems  are consequence of the following.
\begin{theorem}({\sc E. Aubry, J.-F. Grosjean, \cite{AubGro2}})\label{arbre}
There exists a (computable) constant $C=C(m)$ such that, for any compact submanifold $M^m$ of $\R^{n+1}$ and any closed subset $A\subset M$ that intercepts any connected component of $M$, there exists a finite family $(T_i)_{\in I}$ of geodesic trees in $M$ with  $A\cap T_i\neq\emptyset$ for any $i\in I$, $d_H\bigl(A\cup\displaystyle\bigcup_{i\in I} T_i,M\bigr)\leqslant C \bigl(\Vol(M\setminus A)\bigr)^\frac{1}{m}$ and $\displaystyle\sum_{i\in I}\h1(T_i)\leqslant C^{m(m-1)}\displaystyle\int_{M\setminus A}|\H|^{m-1}$.
\end{theorem}
\begin{remark} Note that by construction the $A\cup\displaystyle\bigcup_{i\in I} T_i$ has the same number of connected components than $A$.
\end{remark}

\begin{prooftheoremcurvature}
We set $\partial_r\Omega$ the union of the connected components of $\partial\Omega$ that intercept $A_{\delta(\Omega)^\frac{1}{4}}$ and we apply Theorem \ref{arbre} to the hypersurface $\partial_r\Omega$ and the set
$A_0=\partial\Omega\cap A_{\delta(\Omega)^{1/4}}=\partial_r\Omega\cap A_{\delta(\Omega)^{1/4}}$. We set $T_0$ the union of the trees given by the theorem. Then we get $\h1(T_0)\leqslant C(n)\displaystyle\int_{\partial_r\Omega\setminus A_0}|\H|^{n-1}$, the set $A_{\delta(\Omega)^{1/4}}\cup T_0$ is connected and by the first point of Theorem \ref{Hausdorff} (or Lemma \ref{concentration}) and Theorem \ref{arbre}, we have
\begin{align*}d_H(A_0\cup T_0,\partial_r\Omega)&\leqslant C(n)\hn(\partial_r\Omega\setminus A_0)^{\frac{1}{n}}\leqslant C(n)\hn(\partial\Omega\setminus A_{\delta(\Omega)^{1/4}})^{\frac{1}{n}}\\
&\leqslant C(n)P(\Omega)^{1/n}\delta(\Omega)^{\frac{1}{4n}}\leqslant C(n)R_\Omega\delta(\Omega)^\frac{1}{4n}
\end{align*}
If we now apply Theorems \ref{arbre} and Theorem \ref{Hausdorff} to each connected component $C_i$ of $\partial\Omega\setminus\partial_r\Omega$ with $A_i=\{x_i\}\subset C_i$, we get a connected union of trees $T_i$ such $d_H(T_i,C_i)\leqslant C(n)R_\Omega\delta(\Omega)^\frac{1}{4n}$ and $\h1(T_i)\leqslant C(n)\displaystyle\int_{C_i}|\H|^{n-1}\dhn$. If we set $T=T_0\cup\displaystyle\bigcup_{i\in I}T_i$, then we have $\h1(T)\leqslant C(n)\displaystyle\int_{\partial\Omega\setminus A_{\delta(\Omega)^\frac{1}{4}}}|\H|^{n-1}\dhn$ and
\begin{align*}
d_H\bigl(\partial\Omega,& \,S_{x_\Omega}(R_{\Omega})\cup T\bigr)\leqslant \max \Bigl(d_H\bigl(\partial_r\Omega,S_{x_\Omega}(R_\Omega)\cup T_0\bigr),(d_H(C_i,T_i))_{i\in I}\Bigr)\\
&\leqslant\max\Bigl(d_H(\partial_r\Omega,A_0\cup T_0)+d_H\bigl(A_0\cup T_0,S_{x_\Omega}(R_{\Omega})\cup T_0\bigr),C(n)R_\Omega\delta(\Omega)^\frac{1}{4n}\Bigr)\\
&\leqslant C(n)R_\Omega\delta(\Omega)^\frac{1}{4n}+d_H\bigl(A_0,S_{x_\Omega}(R_\Omega)\bigr)\\
&\leqslant C(n)R_\Omega\delta(\Omega)^{\beta(n)}
\end{align*}
the last inequality comes from Theorem \ref{Hausdorff}. This completes the proof of Theorem \ref{distancefrontspheremodel}.

Now to prove Theorem \ref{Hausdorffcourbure}, we have
\begin{align*}
d_H\left(\partial\Omega, S_{x_\Omega}(R_\Omega)\cup\left(\bigcup_{i\in I}\{x_i\}\right)\right)\\
&\hspace{-3cm}\leqslant d_H\bigl(\partial\Omega,S_{x_\Omega}(R_\Omega)\cup T\bigr)+d_H\left(S_{x_\Omega}(R_\Omega)\cup T,S_{x_\Omega}(R_\Omega)\cup \left(\bigcup_{i\in I}\{x_i\}\right)\right)\\
&\hspace{-3cm}\leqslant C(n)R_\Omega\delta(\Omega)^{\beta(n)}+d_H\left(T,(T_0\cap S_{x_\Omega}(R_\Omega))\cup\left(\bigcup_{i\in I}\{x_i\}\right)\right)\\
&\hspace{-3cm}\leqslant C(n)R_\Omega\delta(\Omega)^{\beta(n)}+\max\bigl(d_H\bigl(T_0,T_0\cap S_{x_\Omega}(R_\Omega)\bigr),(d_H(T_i,\{x_i\}))_{i\in I}\bigr)\\
&\hspace{-3cm}\leqslant C(n)R_\Omega\delta(\Omega)^{\beta(n)}+C(n)\int_{\partial\Omega\setminus A_{\delta(\Omega)^\frac{1}{4}}}|\H|^{n-1}\dhn
\end{align*}
To finish the proof of Theorem \ref{Hausdorffcourbure} we just have to use H\"older's inequality and Lemma \ref{concentration}. 
For what concerns cardinality of $I$, remark that the Michael-Simon Inequality applied to the function $f=1$ and to any connected component $C$ of $\partial\Omega\setminus\partial_r\Omega$ gives us
$$\hn(C)^\frac{n-1}{n}\leqslant C(n)\int_{C}|\H|\dhn\leqslant C(n)\Bigl(\int_{C}|\H|^n\dhn\Bigr)^\frac{1}{n}(\hn(C)^\frac{n-1}{n}$$
and so $\int_C|\H|^n\dhn\geqslant\frac{1}{C(n)}$ for any connected component of $\partial\Omega\setminus\partial_r\Omega$. We infer that 
 $$\frac{\Card(I)}{C(n)}\leqslant\sum_C\int_C|\H|^n\dhn\leqslant\int_{\partial\Omega\setminus A_{\delta(\Omega)^\frac{1}{4}}}|\H|^n\dhn$$
 we conclude for any $p\geqslant n$ by H\"older inequality and Lemma \ref{concentration}.
\end{prooftheoremcurvature}

\subsection{Variants of Theorems \ref{distancefrontspheremodel} and \ref{Hausdorffcourbure} that generalize inequality \eqref{fuglede} }\label{postponed}
\begin{theorem}\label{distanceballmodel} Let $\Omega$ be an open set with  smooth boundary, finite perimeter and $\delta(\Omega)\leqslant\frac{1}{C(n)}$. There exists a subset $T\subset \R^{n+1}$ with 
\begin{enumerate}
\item $\mathcal{H}^1(T)\leqslant C(n)R_\Omega \int_{\partial\Omega\setminus B_{x_\Omega}\bigl(R_\Omega(1+\delta(\Omega)^\frac{1}{4})\bigr)}|\H|^{n-1}\dhn$,
\item $d_H\bigl(\Omega,B_{x_\Omega}(R_{\Omega})\cup T\bigr)\leqslant C(n)R_\Omega\delta(\Omega)^{\beta(n)}$,
\item the set
 $B_{x_\Omega}\bigl(R_\Omega(1+\delta(\Omega)^\frac{1}{4})\bigr)\cup T$ has at most $N+1$ connected components,
\end{enumerate}
where $N$ is the number of connected components of $\partial\Omega$ that do not intercept the ball $B_{x_\Omega}\bigl(R_\Omega(1+\delta(\Omega)^\frac{1}{4})\bigr)$.
\end{theorem}

\begin{theorem}\label{Hausdorffcourbureball} 
 Let $p\geqslant n-1$ and $\Omega$ be an open set with smooth boundary $\partial\Omega$, finite perimeter and $\delta(\Omega)\leqslant\frac{1}{C(n)}$.  Let $(\partial\Omega_i)_{i\in I}$ be the connected components of $\partial\Omega$ that do not intercept the ball $B_{x_\Omega}\bigl(R_\Omega(1+ \delta(\Omega)^\frac{1}{4})\bigr)$. For any $i\in I$, there exists $x_i\in\partial\Omega_i$ such that 
\begin{align*}
d_H\bigl(\Omega, B_{x_\Omega}(R_{\Omega})\cup\bigcup_{i\in I}\{x_i\}\bigr)\leqslant C(n,p)R_\Omega\left[\delta(\Omega)^{\beta(n)}+\delta(\Omega)^{\frac{p-n+1}{4p}}(P(\Omega)^{\frac{1}{n}}\|\H\|_p)^{n-1}\right]
\end{align*}
Moreover if $p\geqslant n$ and $\H$ is $L^p$-integrable then $I$ is of finite cardinal $N$ and we have
$$N\leqslant C(n,p)P(\Omega)\|\H\|_p^n\delta(\Omega)^\frac{p-n}{4p}.$$
\end{theorem}
\begin{remark} The norm  $\|\H\|_p$  can  be  replaced  by  $\left(\frac{1}{P(\Omega)}\displaystyle\int_{\partial\Omega\setminus B_{x_\Omega}\bigl(R_\Omega(1+ \delta(\Omega)^\frac{1}{4})\bigr) }|\H|^{p}\dhn\right)^\frac{1}{p}$.
\end{remark}
\begin{proofnewversion}
We set $\partial_r\Omega$ the union of the connected components of $\partial\Omega$ that intercept $ B_{x_\Omega}\bigl(R_\Omega(1+\delta(\Omega)^\frac{1}{4})$ and then we construct $T$ as in the previous section. Arguing as in the previous subsection, we get that the  $C(n)R_\Omega\delta(\Omega)^{\beta(n)}$-tubular neighbourhood of $B_{x_\Omega}(R_\Omega)\cup T$ contains $\partial \Omega\setminus B_{x_\Omega}\bigl(R_\Omega(1+\delta(\Omega)^\frac{1}{4})\bigr)$. We set with $R'_\Omega=R_\Omega(1+2C(n)\delta(\Omega)^\frac{1}{2(n+1)})$ (where $C(n)$ is the constant of Remark \ref{DomainHausdorff} (2)). Then for any $x\in \Omega$,  either we have $x\in B_{x_\Omega}(R'_{\Omega})$ and then $d\bigl(x, B_{x_\Omega}(R_\Omega)\cup T\bigr)\leqslant 2C(n)R_\Omega \delta(\Omega)^\frac{1}{2(n+1)}$, either we have $x\in \Omega\setminus B_{x_\Omega}(R'_{\Omega})$, and then $x\in\Omega\Delta B_{x_\Omega}(R_{\Omega})$. From the Remark \ref{DomainHausdorff} (1), we infer (as in the proof of Remark \ref{DomainHausdorff} (2)) that
$$d(x,\partial\Omega)\leqslant C(n)R_\Omega\delta(\Omega)^{\frac{1}{2(n+1)}}$$
and even more precisely, $d\bigl(x,\partial\Omega\setminus B_{x_\Omega}\bigl(R_\Omega(1+\delta(\Omega)^\frac{1}{4}))\bigr)\leqslant C(n)R_\Omega\delta(\Omega)^{\frac{1}{2(n+1)}}$. We infer that we have $$d\bigl(x,B_{x_\Omega}(R_\Omega)\cup T\bigr)\leqslant C(n)R_\Omega\delta(\Omega)^{\min(\frac{1}{8},\frac{1}{4n})}$$

On the other hand, for any $x\in B_{x_\Omega}(R_\Omega)\cup T$, either $x\in B{x_\Omega}(R_\Omega)$ and then $d(x,\Omega)\leqslant C(n)R_{\Omega}\delta(\Omega)^\frac{1}{2(n+1)}$ by Remark \ref{DomainHausdorff} (2), either $x\in T$ and then $d(x,\Omega)=0$. We then get
$$d_H\bigl(B_{x_\Omega}(R_\Omega)\cup T,\Omega\bigr)\leqslant C(n)R_\Omega\delta(\Omega)^{\min(\frac{1}{8},\frac{1}{4n})}$$
which gives the result as in the proofs of Theorems \ref{distancefrontspheremodel} and \ref{Hausdorffcourbure}.
\end{proofnewversion}
\section{A quasi-isometry result : proof of Theorem \ref{Lipschitz}}\label{Allard}
Let us first remind Duggan's version of Allard's local regularity theorem about hypersurface of suitably bounded mean curvature.

\begin{theorem}[J.P. Duggan \cite{Dug}]\label{All}
If $p > n$ is arbitrary, then there are $\eta=\eta(n,p)$, $\gamma=\gamma(n,p)\in(0,1)$ and $c = c(n,p)$ such that if $M\subset\R^{n+1}$ is a hypersurface, $x\in M$ and $\rho>0$ satisfy the hypotheses 
\begin{enumerate}
\item $\hn(B_x(\rho)\cap M)\leqslant (1+\eta)\rho^n|\B^n|$\\
\item $\rho^{p-n}\displaystyle\int_{B_x(\rho)\cap M}|\H|^p\dhn\leqslant\eta^p$
\end{enumerate} 
then there exists a linear isometry $q$ of $\R^{n+1}$ and  $u \in W^{2,p}(B_0^{\R^n}(\gamma\rho))$ with $u(0)=0$, $M\cap B_x(\gamma\rho)=\bigl(x+q(graph\ u)\bigr)\cap B_x(\gamma\rho)$ and
\begin{equation}\label{all}
\frac{\sup|u|}{\rho}+\sup |du|+\rho^{1-\frac{n}{p}}\left(\int_{B_0^{\R^n}(\gamma\rho)}|\nabla du|^p\dhn\right)^{1/p}\leqslant c\eta^\frac{1}{4n}.
\end{equation}
\end{theorem}
So  the Morrey-Campanato  says that for any $v\in W^{1,p}(B_0^{\R^n}(1))$ we have 
$$\sup_{x\neq y\in B_0^{\R^n}(1)}\frac{|v_x-v_y|}{|x-y|^{1-\frac{n}{p}}}\leqslant C(n,p)\left(\int_{B_0^{\R^n}(1)}|v|^p\dhn+\int_{B_0^{\R^n}(1)}|dv|^p\dhn\right)$$
Up to a normalization and under the assumptions of Theorem \ref{All}, the Morrey-Campanato theorem gives us that 
$$\rho^{1-\frac{n}{p}}\sup_{x\neq y\in B_0^{\R^n}(\gamma\rho)}\frac{||du_x|-|du_y||}{|x-y|^{1-\frac{n}{p}}}\leqslant C(n,p)\eta^\frac{1}{4n}$$
Now let $\Phi : B_0^{\R^n}(\gamma\rho)\rightarrow\R^{n+1}\ ,\ a\mapsto q(a,u(a))$. Then $d\Phi_a(h)=q(h,du_a(h))$. Since $q$ is an isometry,  a unit normal is given by $\nu_{\Phi(a)}=q\bigl(\frac{((\nabla u)\mid_a,-1)}{\sqrt{1+|\nabla u\mid_a|^2}}\bigr)$ which gives for any $x\in\partial\Omega$
\begin{equation}\label{Holdernu}
\rho^{1-\frac{n}{p}}\sup_{y,z\in B_x(\gamma\rho)\cap\partial\Omega,\ y\neq z}\frac{\bigl|\nu_y-\nu_z\bigr|}{|y-z|^{1-\frac{n}{p}}}\leqslant C(n,p)\eta^\frac{1}{4n}
\end{equation}

\begin{lemma}\label{Lipprep}
Let $p>n$. There exist $3$ positive constants $C_1(n,p)$, $C_2(n,p)$ and $C_3(n,p)$ such that for any domain $\Omega$ with smooth boundary $\partial\Omega$ satisfying $P(\Omega)\normlp{\H}{p}^n\leqslant K$, and $\delta(\Omega)\leqslant\frac{1}{C_1(n,p)K^{\alpha(n,p)}}$, we have
\begin{equation}\label{ineq15}
\sup_{x\in\partial\Omega}\bigl||x-x_{\Omega}|-R_\Omega\bigr|\leqslant C_2(n,p)R_\Omega\delta(\Omega)^{\beta(n)},
\end{equation}
and the assumptions of Theorem \ref{All} are satisfied for $\bar{\rho}=\frac{R_\Omega}{C_3(n,p)K^\frac{p}{n(p-n)}}$. Moreover we have
\begin{equation}\label{ineq13}
\bar{\rho}^{1-\frac{n}{p}}\sup_{y,z\in B_x(\gamma\bar{\rho})\cap\partial\Omega,\ y\neq z}\frac{\bigl|Z_y-Z_z\bigr|}{|y-z|^{1-\frac{n}{p}}}\leqslant C(n,p)\eta^\frac{1}{4n}
\end{equation}
Where $Z_x=\frac{x-x_{\Omega}}{|x-x_{\Omega}|}-\nu_x$. Here we have set $\alpha(n,p)=\frac{8p}{p-n}$.
\end{lemma}

\begin{proof} Since the computations are a bit messy, we organize them in several steps:
\begin{enumerate}[leftmargin=*,parsep=0cm,itemsep=0cm,topsep=0cm]
\item For what concern the point (2) of  Theorem \ref{All}, we have for any $\rho>0$
\begin{align*}
\rho^{p-n}\int_{B_x(\rho)\cap\partial\Omega}|\H|^p\dhn&\leqslant \rho^{p-n}P(\Omega)\|\H\|_p^p\leqslant \left(\frac{\rho}{P(\Omega)^\frac{1}{n}}\right)^{p-n}K^{\frac{p}{n}}
\end{align*}
From \eqref{inegisop} and the definition of $R_{\Omega}$, we have $R_{\Omega}^n\leqslant C(n) P(\Omega)$ and so
\begin{align*}
\rho^{p-n}\int_{B_x(\rho)\cap\partial\Omega}|\H|^p\dhn\leqslant C(n,p)\left(\frac{\rho}{R_\Omega}\right)^{p-n}K^{\frac{p}{n}}
\end{align*}
From this we deduce that there exists a constant $C_3(n,p)$ large enough such that $\partial\Omega$ satisfies assumption (2) of Theorem \ref{All} for $\rho=\bar{\rho}=\frac{R_\Omega}{C_3(n,p)K^\frac{p}{n(p-n)}}$.

\item Let $x\in\S^n$ then there exists a $r(n,p)\in]0,1]$ such that we have $\frac{\hn(B_x(r)\cap\S^n)}{|\B^n|r^n}\in[1/2,1+\eta(n,p)/4]$, for any $r<r(n,p)$, where $\eta(n,p)$ is the constant of Theorem \ref{All}. By Michael-Simon Sobolev inequality, we have $K\geqslant P(\Omega)\|\H\|_p^n\geqslant k(n)$, and so we can assume $C_3(n,p)$ large enough to have $\bar{\rho}/R_{\Omega}\leqslant r(n,p)\leqslant 1$.

From now on $C_3(n,p)$ is fixed so that it satisfies both the two previous conditions.

\item  Since $K\geqslant k(n)$, we can assume $C_1(n,p)$ large enough for $\delta(\Omega)\leqslant\frac{1}{C_1(n,p)K^{\alpha(n,p)}}$ to imply that $\delta(\Omega)\leqslant\min\bigl(\frac{\eta}{\eta+4},(\frac{|\B^n|(C(n))^{n-1}\eta}{8\hn(\S^n)})^8\bigr)\leqslant 1$ in what follows, where $C(n)$ is the constant of Theorem \ref{equidense}.

\item From Theorem \ref{Hausdorffcourbure},  the number $N$ of connected components of $\partial\Omega$ that do not intercept $A_{\delta(\Omega)^{1/4}}$ satisfies
$$N\leqslant C(n,p)P(\Omega)\|\H\|_p^n\delta(\Omega)^\frac{p-n}{4p}\leqslant C(n,p)K\delta(\Omega)^{\frac{p-n}{4p}}\leqslant C(n,p)K\delta(\Omega)^\frac{1}{\alpha(n,p)}$$
So, when $\delta(\Omega)\leqslant\frac{1}{(2C(n,p)K)^{\alpha(n,p)}}$, we have $N=0$. We infer by Theorem \ref{Hausdorffcourbure}
\begin{align*}d_H\bigl(\partial\Omega,S_{x_\Omega}(R_\Omega)\bigr)&\leqslant C(n,p)R_\Omega\left[\delta(\Omega)^{\beta(n)}+\delta(\Omega)^{\frac{p-n+1}{4p}}(P(\Omega)^{\frac{1}{n}}\|\H\|_p)^{n-1}\right]\\
&\leqslant C(n,p)R_\Omega\left(1+\delta(\Omega)^{\left(\frac{n-1}{n}\right)\left(\frac{p-n}{4p}\right)}K^{\frac{n-1}{n}}\right)\delta(\Omega)^{\beta(n)}\\
&\leqslant C(n,p)R_\Omega\Bigl(1+\bigl(\delta(\Omega)^\frac{1}{\alpha(n,p)}K\bigr)^{\frac{n-1}{n}}\Bigr)\delta(\Omega)^{\beta(n)}\\
&\leqslant C_2(n,p)R_\Omega\delta(\Omega)^{\beta(n)}
\end{align*}
which gives inequality \eqref{ineq15} for any $C_1(n,p)\geqslant (2C(n,p))^{\alpha(n,p)}$ such that the previous condition (3) also holds. Note that we have used $\delta(\Omega)\leqslant 1$. At this stage, $C_2(n,p)$ is fixed, and does not depends on $C_1(n,p)$.

\item Similarly for $C_1(n,p)$ large enough and $\delta(\Omega)\leqslant\frac{1}{C_1K^{\alpha}}$, we have from the previous point that
\begin{equation}\label{minorx}|x|\geqslant R_\Omega(1-C_2\delta(\Omega)^{\beta})\geqslant R_\Omega\left(1-\frac{C_2}{C_1^{\beta}K^{\beta\alpha}}\right)\geqslant R_\Omega\left(1-\frac{C_2}{C_1^{\beta}k(n)^{\beta\alpha}}\right)\geqslant\frac{1}{2}R_{\Omega}
\end{equation}
From this we deduce that 
$$\frac{\left|\frac{x-x_{\Omega}}{|x-x_{\Omega}|}-\frac{y-x_{\Omega}}{|y-x_{\Omega}|}\right|}{|x-y|^{1-\frac{n}{p}}}\leqslant\frac{4}{R_{\Omega}}|x-y|^{\frac{n}{p}}\leqslant C(n,p) R_{\Omega}^{\frac{n}{p}-1}\leqslant \frac{1}{\bar{\rho}^{1-\frac{n}{p}}C_3^{\frac{p-n}{p}}K^{1/n}}\leqslant \frac{C(n,p)}{\bar{\rho}^{1-\frac{n}{p}}}$$
which gives with \ref{Holdernu} the inequality \ref{ineq13}.
\item We want to apply Theorem \ref{equidense} to $\partial\Omega$ and $B_x(\bar{\rho})$ and so need $\bar{\rho}\in[C(n)\delta(\Omega)^{1/8}R_{\Omega},R_{\Omega}]$. Note that $\bar{\rho}\leqslant R_\Omega$ was already obtained in (2). On the other hand,  we have that
\begin{align*}
\frac{\bar{\rho}}{R_\Omega}=\frac{C(n)\delta(\Omega)^\frac{1}{8n}}{C(n)\delta(\Omega)^\frac{1}{8n}C_3(n,p)K^\frac{p}{n(p-n)}}&=\frac{C(n)\delta(\Omega)^\frac{1}{8n}C_1(n,p)^\frac{1}{8n}}{C(n)C_3(n,p)[\delta(\Omega)C_1(n,p)K^{\alpha(n,p)}]^\frac{1}{8n}}\\
&\geqslant C(n)\delta(\Omega)^\frac{1}{8}\frac{C_1(n,p)^\frac{1}{8}}{C(n)C_3(n,p)}.
\end{align*}
 Now it is clear that for $C_1(n,p)$ large enough, we have $\frac{\bar{\rho}}{R_\Omega}\geqslant C(n)\delta(\Omega)^\frac{1}{8n}\geqslant C(n)\delta(\Omega)^\frac{1}{8}$.

\item  Now we prove that for $C_1(n,p)$ large enough, $\partial\Omega$ satisfies (1) for $\bar{\rho}=\frac{R_\Omega}{C_3(n,p)K^\frac{p}{n(p-n)}}$ with $C_3(n,p)$ fixed in (2).
Let $x\in S_{x_{\Omega}}(R_{\Omega})$. Then Theorem \ref{equidense} gives us
\begin{align*}\frac{\hn\bigl(B_x(\bar{\rho})\cap S_{x_\Omega}(R_\Omega)\bigr)}{P(\Omega)}&\leqslant C(n)\delta(\Omega)^{1/4}+\frac{\hn(B_x(\bar{\rho})\cap S_{x_{\Omega}}(R_{\Omega})}{R_{\Omega}^n\hn(\S^n)}\\
&\leqslant C(n)\delta(\Omega)^{1/4}+\frac{\hn\Bigl(B_{x'}\bigl(\frac{\bar{\rho}}{R_{\Omega}}\bigr)\cap S_{x_{\Omega}}(1)\Bigr)}{\hn(\S^n)}
\end{align*}
where $x'=x_{\Omega}+\frac{1}{R_{\Omega}}(x-x_{\Omega})$. Now by the condition (2) above, we have $$\hn\Bigl(B_{x'}\bigl(\frac{\bar{\rho}}{R_{\Omega}}\bigr)\cap S_{x_{\Omega}}(1)\Bigr)\leqslant(1+\eta/4)|\B^n|\frac{\bar{\rho}^n}{R_\Omega^n}$$ and so
\begin{align*}
\hn\bigl(B_x(\bar{\rho})\cap S_{x_\Omega}(R_\Omega)\bigr)&\leqslant C(n) P(\Omega)\delta(\Omega)^{1/4}+(1+\eta/4)\frac{P(\Omega)|\B^n|\bar{\rho}^n}{\hn(\S^n)R_{\Omega}^n}\notag\\
&\leqslant\Bigl(\frac{2C(n)\hn(\S^n)R_{\Omega}^n}{|\B^n|\bar{\rho}^n}\delta(\Omega)^{1/4}+(1+\eta/4)\bigl(1+\delta(\Omega)\bigr)\Bigr)|\B^n|\bar{\rho}^n\\
&\leqslant\Bigl(\frac{2\hn(\S^n)}{(C(n))^{n-1}|\B^n|}\delta(\Omega)^{1/8}+(1+\eta/4)\bigl(1+\delta(\Omega)\bigr)\Bigr)|\B^n|\bar{\rho}^n\end{align*}
where we have used the fact that $\frac{P(\Omega)}{R_{\Omega}^n\hn(\S^n)}=1+\delta(\Omega)\leqslant 2$, and $\dfrac{R_{\Omega}^n}{\bar{\rho}^n}\leqslant\dfrac{1}{C(n)^n\delta(\Omega)^{1/8}}$ proved in (5). Now from the condition $\delta(\Omega)\leqslant\min\bigl(\frac{\eta}{\eta+4},(\frac{|\B^n|(C(n))^{n-1}\eta}{8\hn(\S^n)})^8\bigr)$ we deduce that
$$\hn\bigl(B_x(\bar{\rho})\cap S_{x_\Omega}(R_\Omega)\bigr)\leqslant\left(\eta/4+(1+\eta/4)(1+\frac{\eta}{\eta+4}\right)|\B^n|\bar{\rho}^n\leqslant(1+\eta)|\B^n|\bar{\rho}^n$$

\end{enumerate} 
\end{proof}

Now, using Duggan's regularity theorem, we can show a Calderon-Zygmund property of almost isoperimetric manifolds with $L^p$ bounded mean curvature:

\begin{lemma}\label{SFBound}
Let $p>n$. There exists $C(n,p)>0$ such that for any domain $\Omega$ with smooth boundary $\partial\Omega$ satisfying $P(\Omega)\normlp{\H}{p}^n\leqslant K$ and $\delta(\Omega)\leqslant\frac{1}{C(n,p)K^{\alpha(n,p)}}$ we have
$$P(\Omega)\|\Bf\|_p^n\leqslant  C(n,p)K^\frac{p+1}{p-n}$$
\end{lemma}

\begin{remark}
We can improve the proof below to get $P(\Omega)\|\Bf\|_p^n\leqslant  C(n,p)K^\frac{p}{p-n}$.
\end{remark}

\begin{proof}
Let $(x_i)_i$ be a maximal family of points of $\partial\Omega$ such that the balls $B_{x_i}(\gamma\bar{\rho}/2)$ are disjoints in $\R^{n+1}$. Then the family $\bigl(\partial\Omega\cap B_{x_i}(\gamma\bar{\rho})\bigr)_i$ covers $\partial\Omega$. By \eqref{ineq15}, all the balls $B_{x_i}(\gamma\bar{\rho}/2)$ are included in $B_{x_i}\left(\left(\frac{\gamma}{2C_3k^{\frac{p}{p-n}}}+\frac{C_2}{C_1^{\beta}k^{\alpha\beta}}+1\right)R_{\Omega}\right)$ and for $C_1$ and $C_3$ large enough, they are included in $B_{x_\Omega}(3R_\Omega)$. And so the family has at most $(\frac{6R_\Omega}{\gamma\bar{\rho}})^{n+1}\leqslant	C(n,p)K^\frac{(n+1)p}{n(p-n)}$  elements (note that using the fact that $\partial\Omega$ is Hausdorff close to $S_{x_\Omega}(R_\Omega)$ we could replace $K^\frac{(n+1)p}{n(p-n)}$ by the better $K^\frac{p}{p-n}$).

By Theorem \ref{All}, denoting by $u_i$ each corresponding function we then have  $|B|\leqslant\sqrt{n}\frac{|d^2u_i|}{\sqrt{1+|du_i|^2}}$ on $\partial\Omega\cap B_{x_i}(\gamma\bar{\rho})$ and
\begin{align*}
\int_{\partial\Omega\cap B_{x_i}(\gamma\bar{\rho})}|\Bf|^p\dhn&\leqslant \int_{B_0^{\R^n}(\gamma\bar{\rho})}n^{p/2}\frac{|d^2u_i|^p}{(1+|du_i|^2)^\frac{p-1}{2}}\dhn\\
&\leqslant \int_{B_0^{\R^n}(\gamma\bar{\rho})}n^{p/2}|d^2u_i|^p\dhn\leqslant \frac{C(n,p)}{\bar{\rho}^{p-n}}
\end{align*}
from which we get
\begin{align*} P(\Omega)\|\Bf\|_p^n&=P(\Omega)^{1-\frac{n}{p}}\left(\int_{\partial\Omega}|\Bf|^p\dhn\right)^{n/p}\leqslant C(n,p)\left(\frac{P(\Omega)}{\bar{\rho}^n}\right)^{\frac{p-n}{p}}K^{\frac{n+1}{p-n}}\\
&=C(n,p)\left(\frac{P(\Omega)}{R_{\Omega}^n}C_3^nK^{\frac{p}{p-n}}\right)^{\frac{p-n}{p}}K^{\frac{n+1}{p-n}}\leqslant C(n,p)K^{\frac{p+1}{p-n}}
\end{align*}
\end{proof}

Using Duggan's Theorem we now improve the $L^2$ smallness of $Z$ given by Lemma \ref{ZL2} in an $L^\infty$ one.

\begin{lemma}\label{Lipprep2}
Let $p>n$. There exists $C(n,p)>0$ such that for any domain $\Omega$ with smooth boundary $\partial\Omega$ satisfying $P(\Omega)\normlp{\H}{p}^n\leqslant K$ and $\delta(\Omega)\leqslant\frac{1}{C(n,p)K^{\alpha(n,p)}}$, we have
\begin{align}
\sup_{x\in\partial\Omega}|Z_x|&\leqslant C(n,p)K^\frac{1}{n}\delta(\Omega)^\frac{1}{n\alpha(n,p)}\label{ineq16}
\end{align}
Here $\alpha(n,p)$ is the same as in Lemma \ref{Lipprep}.
\end{lemma}

\begin{proof} Let 
\begin{equation}\label{conditionC4}
C_4(n,p)=\max\bigl(\frac{C(n)}{\gamma},\frac{C_2(n,p)}{\gamma},\frac{1}{\gamma}\left(\frac{4C(n)}{|\B^n|}\right)^\frac{1}{n}\bigr)
\end{equation}
where $C(n)$ is the constant of Theorem \ref{equidense} and $C_2(n,p)$ is the constant of Lemma \ref{Lipprep}.  We set $\rho'=2\gamma C_4(n,p)\delta(\Omega)^\frac{1}{8n}R_\Omega$. 

Assume now that $\delta(\Omega)\leqslant\frac{1}{C_1'(n,p)K^{\alpha(n,p)}}$ where $C_1'\geqslant C_1$. For $C_1'(n,p)$ large enough we have $\delta(\Omega)\leqslant\frac{1}{C_1'(n,p)K^{\alpha(n,p)}}\leqslant\frac{1}{(2C_4C_3)^{8n}K^{\alpha(n,p)}}$ and $\rho'\leqslant \gamma\bar{\rho}$. As explained in the point (2) of the proof of Lemma \ref{Lipprep}, we can assume $C_1'(n,p)$ large enough to get that 
\begin{equation}\label{conditionC1}
\delta(\Omega)^\frac{1}{8n}\leqslant\min\bigl(\frac{\gamma C_4}{C_2},\frac{1}{\gamma C_4}\bigr)
\end{equation}
where $C_3(n,p)$ is the constant used in the proof of Lemma \ref{Lipprep}. For any $x\in\partial \Omega$ and for any $y,z\in\partial\Omega\cap B_x(\rho')$, Inequality \eqref{ineq13} and the value of $\bar{\rho}$ give us
\begin{align*}|Z_y-Z_z|&\leqslant\frac{C(n,p)\eta(n,p)^{1/4n}}{\bar{\rho}^{\frac{p-n}{p}}}|y-z|^{1-\frac{n}{p}}\leqslant C(n,p)K^{1/n}\left(\frac{\rho'}{R_{\Omega}}\right)^{1-\frac{n}{p}}
\end{align*}
Since, $K\geqslant k(n)$, we can assume $C_1'(n,p)$ large enough so that Lemma \ref{ZL2} applies and then for any $x\in\Omega$
\begin{align*}
|Z_x|&\leqslant\frac{1}{\hn( B_x(\rho')\cap\partial\Omega)}\Bigl(\int_{B_x(\rho')\cap\partial\Omega}\bigl|Z_x-Z_y\bigr|\dhn(y)+\int_{B_x(\rho')\cap\partial\Omega}\bigr|Z_y\bigl|\dhn(y)\Bigr)\\
&\leqslant C(n,p)K^\frac{1}{n}\bigl(\frac{\rho'}{R_\Omega}\bigr)^{1-\frac{n}{p}}+\Bigl(\frac{1}{\hn( B_x(\rho')\cap\partial\Omega)}\int_{B_x(\rho')\cap\partial\Omega}|Z_y|^2d\hn(y)\Bigr)^{1/2}\\
&\leqslant C(n,p)K^\frac{1}{n}\left(\frac{\rho'}{R_\Omega}\right)^{1-\frac{n}{p}}+C(n)\left(\frac{P(\Omega)}{\hn( B_x(\rho')\cap\partial\Omega)}\right)^\frac{1}{2}\delta(\Omega)^\frac{1}{4}
\end{align*}
Now let $x'=x_{\Omega}+R_{\Omega}\frac{x-x_{\Omega}}{|x-x_{\Omega}|}\in S_{x_{\Omega}}(R_{\Omega})$. From \eqref{ineq15}, an easy computation shows that $B_{x'}\bigl(\frac{\rho'}{2}\bigr)\subset B_x(\rho')$. Indeed if $y\in B_{x'}\bigl(\frac{\rho'}{2}\bigr)$, then 
$$|x-y|\leqslant||x-x_{\Omega}|-R_{\Omega}|+\frac{\rho'}{2}\leqslant C_2R_{\Omega}\delta(\Omega)^{\beta}+\frac{\rho'}{2}$$
From the choices made in \eqref{conditionC4} and \eqref{conditionC1} we have $\delta(\Omega)^{\beta}\leqslant\frac{\gamma C_4}{C_2}\delta(\Omega)^{1/8n}$ and $|x-y|\leqslant\rho'$.
We then get

\begin{equation}\label{majorZ}
|Z_x|\leqslant C(n,p)K^{1/n}\delta(\Omega)^\frac{1}{n\alpha}+C(n)\left(\frac{P(\Omega)}{\hn( B_{x'}(\rho'/2)\cap\partial\Omega)}\right)^\frac{1}{2}\delta(\Omega)^\frac{1}{4}
\end{equation}

Now \eqref{conditionC4} and \eqref{conditionC1} imply that $\delta(\Omega)^{1/8n}\leqslant 1/\gamma C_4$ and $C_4\geqslant C(n)/\gamma$ which gives  $\frac{\rho'}{2}\in[C(n)\delta(\Omega)^{1/8}R_{\Omega},R_{\Omega}]$. So  we can apply Theorem \ref{equidense}, and since we have $\frac{\rho'}{R_\Omega}\leqslant r(n,p)$ (see (2) in the proof of the previous lemma), we get $\hn\left(B_x\left(\frac{\rho'}{2R_{\Omega}}\right)\cap\S^n\right)\geqslant\frac{|\B^n|}{2}\bigl(\frac{\rho'}{2R_{\Omega}}\bigr)^n$ and
\begin{align*}
\frac{\hn(B_{x'}\left(\frac{\rho'}{2}\right)\cap\partial\Omega)}{P(\Omega)}&\geqslant\frac{\hn(B_{x'}\left(\frac{\rho'}{2}\right)\cap S_{x_{\Omega}}(R_{\Omega}))}{R_{\Omega}^n|\S^n|}-C(n)\delta(\Omega)^{1/4}\\
&\geqslant\frac{\hn(B_{x''}\left(\frac{\rho'}{2R_{\Omega}}\right)\cap S_{x_{\Omega}}(1))}{|\S^n|}-C(n)\delta(\Omega)^{1/4}\\
&\geqslant \frac{|\B^n|}{2} \left(\frac{\rho'}{2R_{\Omega}}\right)^n-C(n)\delta(\Omega)^{1/4}\\
&=\frac{|\B^n|(\gamma C_4(n,p))^n}{2}\delta(\Omega)^\frac{1}{8}-C(n)\delta(\Omega)^{1/4}\\
&\geqslant C(n)\delta(\Omega)^{1/8}
\end{align*}
where $x''=x_{\Omega}+\frac{1}{R_{\Omega}}(x'-x_{\Omega})$ and in the last inequality we used again \eqref{conditionC4}.
Reporting this in \eqref{majorZ} we obtain
\begin{align*}|Z_x|&\leqslant C(n,p)K^{1/n}\delta(\Omega)^\frac{1}{n\alpha(n,p)}+C(n,p)\delta(\Omega)^{3/16}\\
&\leqslant C(n,p)K^{1/n}\delta(\Omega)^\frac{1}{n\alpha(n,p)}+\frac{C(n,p)}{k(n)^{1/n}}K^{1/n}\delta(\Omega)^\frac{1}{n\alpha(n,p)}\\
&\leqslant C_5(n,p)K^{1/n}\delta(\Omega)^\frac{1}{n\alpha(n,p)}
\end{align*}
which gives the desired inequality by putting $C(n,p)=\max(C_1'(n,p),C_5(n,p))$.
\end{proof}

Since we have an upper bound on the second fundamental form, we could also perform a Moser iteration as in \cite{AubGro2} to prove the previous lemma.

Let $\Omega$ be an almost isoperimetric  domain. We consider the map $F:\partial\Omega\longrightarrow S_{x_\Omega}(R_{\Omega})$ defined by
$$F(x)=R_{\Omega}\frac{x-x_\Omega}{|x-x_\Omega|}$$

\begin{prooflipschitz} In this proof, $C(n,p)$ is the constant of the Lemma \ref{Lipprep2}. For more convenience up to a translation we can assume $x_{\Omega}=0$. Under the assumptions of Lemma \ref{Lipprep}, we have $|x|\geqslant\frac{1}{2}R_{\Omega}$.
Hence $F$ is well defined on $\partial\Omega$. Moreover, for any $x\in\partial\Omega$ and $u\in T_x\partial\Omega$, we have $dF_x(u)=\frac{R_{\Omega}}{|x|}\left(u-\frac{\langle x,u\rangle}{|x|^2}x\right)=\frac{R_{\Omega}}{|x|}\left(u-\langle Z_x,u\rangle\frac{x}{|x|}\right)$ and we have
$$|dF_x(u)|^2=\frac{R_{\Omega}^2}{|x|^2}(|u|^2-\langle Z_x,u\rangle^2)$$
Let $D(n,p)\geqslant C(n,p)$ large enough and assume $\delta(\Omega)^{1/2}\leqslant\frac{1}{DK^{\alpha}}$. Since by Inequality \eqref{ineq16} of Lemma \ref{Lipprep2} we have 
$$|Z_x|\leqslant CK^{1/n}\delta(\Omega)^\frac{1}{n\alpha}\leqslant\frac{CK^{1/n}}{{D}^\frac{1}{n\alpha} K^\frac{1}{n}}=\frac{C}{{D}^\frac{1}{n\alpha}}$$
Hence we can assume $\|Z\|_\infty< 1/2$ for $D(n,p)$ large enough,  which infer that $F$ is a local diffeomorphism form $\partial\Omega$ into $S_{0}(R_\Omega)$. Let $\partial\Omega_0$ be a connected component of $\partial\Omega$. Since $\partial\Omega_0$ is compact and $S_{0}(R_\Omega)$ is simply connected, we get that $F$ is a diffeomorphism. Moreover since $||x|-R_{\Omega}|\leqslant C_2 R_{\Omega}\delta(\Omega)^{\beta}$ we have $\left|\frac{R_{\Omega}}{|x|}-1\right|\leqslant 2C_2\delta(\Omega)^{\beta}$ and 
\begin{align}\label{quasiisometry}||dF_x(u)|^2-|u|^2|&\leqslant\left|\frac{R_{\Omega}}{|x|}-1\right|\left|\frac{R_{\Omega}}{|x|}+1\right||u|^2+\frac{R_{\Omega}^2}{|x|^2}|Z_x|^2|u|^2\notag\\
&\leqslant\left(6C_2\delta(\Omega)^{\beta}+2\|Z\|_{\infty}\right)|u|^2\notag\\
&\leqslant\left(\frac{6C_2}{k(n)^{1/n}}\delta(\Omega)^{\beta-\frac{1}{n\alpha}}+2C\right)K^{1/n}\delta(\Omega)^\frac{1}{n\alpha}|u|^2\notag\\
&\leqslant\left(\frac{6C_2}{k(n)^{1/n}}+2C\right)\frac{\delta(\Omega)^\frac{1}{2n\alpha}}{D^{1/n\alpha}}|u|^2\notag\\
&\leqslant C_6(n,p)\delta(\Omega)^\frac{1}{2n\alpha}|u|^2
\end{align}
Now if $\partial\Omega$ as at least $2$ connected components $\partial\Omega_0$ and $\partial\Omega_1$ we have for any $i\in\{0,1\}$
\begin{align*}
\hn(S_{0}(R_{\Omega}))=\int_{\partial\Omega_i}F^{\star}\dhn=\int_{\partial\Omega_i}\frac{|\langle x,\nu_x\rangle|}{|x|}\left(\frac{R_{\Omega}}{|x|}\right)^n\dhn\leqslant\frac{\hn(\partial\Omega_i)}{(1-C_2\delta(\Omega)^{\beta})^n}
\end{align*}
and
\begin{align*}
P(\Omega)\geqslant\hn(\partial\Omega_0)+\hn(\partial\Omega_1)\geqslant2(1-C_2\delta(\Omega)^{\beta})^n\hn(S_{0}(R_{\Omega}))\geqslant2\frac{(1-C_2\delta(\Omega)^{\beta})^n}{1+\delta(\Omega)}P(\Omega)
\end{align*}
Where we have used the fact that $\frac{\hn(\partial\Omega)}{\hn(S_0(R_{\Omega}))}=\frac{I(\Omega)}{I(\B^{n+1})}=1+\delta(\Omega)$. Now we can prove easily that $\frac{(1-C_2\delta(\Omega)^{\beta})^n}{1+\delta(\Omega)}>1/2$ for $D$ great enough and we deduce that $\partial\Omega$ has one connected component.

Actually \ Inequality \ref{quasiisometry}\ gives\ for\ $D$\ great\ enough\ that\ $d_L(\partial\Omega,S_{0}(R_{\Omega}))\leqslant C_6(n,p)\delta(\Omega)^\frac{1}{2n\alpha}$ $=C_6(n,p)\delta(\Omega)^{\frac{16pn}{p-n}}$. But we can improve this bound in order to have sharp estimates with respect to the powers of  $\delta(\Omega)$ involved in the estimates on $d_L$ and $d_H$.

Let $\varphi:S_{0}(R_\Omega)\to\R$ given by $\varphi(w)=\|F^{-1}(w)\|/R_\Omega$. Then we have $\partial\Omega=\{\varphi(w)w,w\in S_{0}(R_\Omega)\}$, $\varphi\geqslant 1/2$ and from \ref{ineq15} $\|\varphi-1\|_{\infty}\leqslant C_2\delta(\Omega)^{\beta}$. 
Moreover for any $u\in T_w S_0(R_{\Omega})$ we have :
$$d\varphi_w(u)=\frac{1}{R_{\Omega}^2}\langle dF_{w}^{-1}(u),w\rangle=\frac{\langle dF_{w}^{-1}(u),Z_{F^{-1}(w)}\rangle}{R_{\Omega}}$$
Consequently $R_{\Omega}|d\varphi_w|\leqslant|dF^{-1}_w||Z_{F^{-1}(w)}|$ and for $D$ great enough we deduce from \ref{quasiisometry} that $|dF^{-1}_w|^2\leqslant\frac{1}{1-C_6\delta(\Omega)^{1/2n\alpha}}\leqslant\frac{1}{2}$ and from \ref{ineq16} we get 
$$R_{\Omega}\|d\varphi\|_{\infty}\leqslant C(n,p)K^{1/n}\delta(\Omega)^{1/n\alpha}$$
Now the second fundamental form $\Bf$ of the boundary can be expressed by the formulae 
$$(F^{-1})^{\star}\Bf=\frac{\varphi\nabla d\varphi-d\varphi\otimes d\varphi-\frac{1}{R_{\Omega}^2}(F^{-1})^{\star}g}{\sqrt{|d\varphi|^2+\frac{\varphi^2}{R_{\Omega}^2}}}$$
which gives
\begin{align*}
|\nabla d\varphi|&\leqslant\frac{1}{\varphi}\left(\sqrt{\|d\varphi\|_{\infty}^2+\frac{\|\varphi\|^2_{\infty}}{R_{\Omega}^2}}|(F^{-1})^{\star}\Bf|+\|d\varphi\|_{\infty}^2+\frac{1}{R_{\Omega}^2}|(F^{-1})^{\star}g|\right)\\
&\leqslant\frac{1}{\varphi}\left(\sqrt{\|d\varphi\|_{\infty}^2+\frac{\|\varphi\|^2_{\infty}}{R_{\Omega}^2}}|\Bf\circ F^{-1}||dF^{-1}|^2+\|d\varphi\|_{\infty}^2+\frac{1}{R_{\Omega}^2}|dF^{-1}|^2\right)\\
&\leqslant\frac{1}{2\varphi}\left(\sqrt{\|d\varphi\|_{\infty}^2+\frac{\|\varphi\|^2_{\infty}}{R_{\Omega}^2}}|\Bf\circ F^{-1}|+2\|d\varphi\|_{\infty}^2+\frac{1}{R_{\Omega}^2}\right)\\
&\leqslant\frac{C(n,p,K)}{R_{\Omega}}\left(|\Bf\circ F^{-1}|+\frac{1}{R_{\Omega}}\right)
\end{align*}
On the other hand
\begin{align*}
\|\Bf\circ F^{-1}\|_p^p&=\frac{1}{\hn(S_0(R_{\Omega}))}\int_{S_0(R_{\Omega})}|\Bf\circ F^{-1}|^p d\hn=\frac{1}{\hn(S_0(R_{\Omega}))}\int_{\partial\Omega}|\Bf|^pF^{\star}d\hn\\
&\leqslant\frac{1}{\hn(S_0(R_{\Omega}))}\int_{\partial\Omega}|\Bf|^p\frac{|\langle x,\nu_x\rangle|}{|x|}\left(\frac{R_{\Omega}}{|x|}\right)^nd\hn\leqslant\frac{2^n\hn(\partial\Omega)}{\hn(S_0(R_{\Omega}))}\|\Bf\|_p^p
\end{align*}
Now $\frac{\hn(\partial\Omega)}{\hn(S_0(R_{\Omega}))}=\frac{I(\Omega)}{I(\B^{n+1})}=1+\delta(\Omega)\leqslant 2$ which gives with  Lemma \ref{SFBound} and the fact that $P(\Omega)^{1/n}=R_{\Omega}|\B^{n+1}|^{1/n+1}I(\Omega)^{1/n}$ 
\begin{align*}
\|\nabla d\varphi\|_p\leqslant C(n,p,K)\left(\frac{1}{R_{\Omega}^2}+\frac{1}{R_{\Omega}P(\Omega)^{1/n}}\right)\leqslant\frac{C(n,p,K)}{R_{\Omega}^2}
\end{align*}
If we set $u:\S^n\to\R$ defined by $u(x)=\varphi(R_\Omega x)-1$ we have for $D$ large enough $\|u\|_{\infty}=\|\varphi-1\|_{\infty}\leqslant C_2\delta(\Omega)^{\beta}\leqslant\frac{3}{20(n+1)}$ and $\|du\|_{\infty}=R_{\Omega}\|d\varphi\|_{\infty}\leqslant CK^{1/n}\delta(\Omega)^{1/n\alpha}\leqslant 1/2$ and so $\Omega$ is a nearly spherical domain in the sense of Fuglede. 

Moreover since $\|\nabla du\|_p=R_{\Omega}^2\|\nabla d\varphi\|_p\leqslant C(n,p,K)$ we have $\|du\|_{W^{1,p}}\leqslant C(n,p,K)$ and by the Campanato-Morrey estimate, we then get for any $x,x_0\in\S^n$ that
$$\frac{||du_x|-|du_{x_0}||}{d_{\S^n}(x,x_0)^{1-\frac{n}{p}}}\leqslant C(n,p,K)$$
and choosing $x_0$ such that $|du_{x_0}|=\|du\|_\infty$ we have
$$|du_x|\geqslant \|du\|_\infty-C(n,p,K)(d_{\S^n}(x,x_0))^{1-\frac{n}{p}}$$
Let $r_0:=\bigl(\frac{\|d u\|_\infty}{C(n,p,K)}\bigr)^ \frac{1}{1-\frac{n}{p}}$. We can assume $r_0<\frac{\pi}{2}$ by taking $C(n,p,K)$ large enough. Integrating the above inequality on the ball of $\S^n$ of center $x_0$ and radius $r_0$ and using the estimates of \cite{Fug1} and then Inequality (I.a) of \cite{Fug1} we get that
\begin{align}\label{gradu}
10\delta(\Omega)&\geqslant\|du\|_2^2\geqslant\frac{1}{|\S^n|}\int_{B_{x_0}^{\S^n}(r_0)}|du|^2d\hn\notag\\
&\geqslant\frac{|\S^{n-1}|}{|\S^n|}\int_{0}^{r_0}(\|du\|_{\infty}-C(n,p,K)t^{1-\frac{n}{p}})^2\sin^{n-1}tdt\notag\\
&\geqslant\left(\frac{2}{\pi}\right)^{n-1}\frac{|\S^{n-1}|}{|\S^n|}\int_{0}^{r_0}(\|du\|_{\infty}-C(n,p,K)t^{1-\frac{n}{p}})^2t^{n-1}tdt\notag\\
&\geqslant\frac{1}{C'(n,p,K)}\|du\|_{\infty}^{2+\frac{n}{1-\frac{n}{p}}}
\end{align}

From which we infer that 
\begin{equation}\label{gradf}
(R_\Omega\|d\varphi\|_\infty)^{2+\frac{n}{1-\frac{n}{p}}}=\|du\|_{\infty}^{2+\frac{n}{1-\frac{n}{p}}}\leqslant C(n,p,K)\delta(\Omega)
\end{equation}
Using Inequalities (I.b) of \cite{Fug1} and the above inequality \eqref{gradf}, we get that
\begin{align*}d_H(\partial\Omega,S_{x_\Omega}(R_\Omega))&=R_{\Omega}\|u\|_\infty\leqslant R_{\Omega}C(n)\|du\|_{\infty}^{\frac{n-2}{n}}\delta(\Omega)^{1/n}\\
&\leqslant C(n,p,K)R_\Omega\delta(\Omega)^\frac{2p-n}{2p-2n+np}
\end{align*}
for $n\geqslant 3$
and $d_H(\partial\Omega,S_{x_\Omega}\leqslant C(p,K)R_{\Omega} (-\delta(\Omega)\ln\delta(\Omega))^\frac{1}{2}$ for $n=2$.

Now since for any $x\in\partial\Omega$, $F^{-1}(x)=x\varphi(x)$, and so $|dF_{R_\Omega x}^{-1}(v)|^2=(1+u(x))^2|v|^2+(du_x(v))^2$, we can use the previous estimates on $u$ to obtain
$$||dF_{R_\Omega x}^{-1}(v)|^2-|v|^2|\leqslant(\|du\|_{\infty}^2+2\|u\|_{\infty}+\|u\|_{\infty}^2)|v|^2\leqslant C(n,p,K)\delta(\Omega)^{\frac{2p-2n}{2p-2n+np}}|v|^2$$
From this and the definition of the Lipschitz distance we conclude that for $\delta(\Omega)$ small enough
\begin{align*}
d_L(\partial\Omega,S_{x_{\Omega}}(R_{\Omega}))&\leqslant(|\ln\dil(F)|+|\ln\dil(F^{-1})|)\\
&\leqslant\max(|\ln(1+C\delta(\Omega)^{\frac{2p-2n}{2p-2n+np}}|,|\ln(1-C\delta(\Omega)^{\frac{2p-2n}{2p-2n+np}})|)\\
&\leqslant C(n,p,K)\delta(\Omega)^{\frac{2p-2n}{2p-2n+np}}
\end{align*}
where for any diffeomorphism $f$ from $\partial\Omega$ into $S_{0}(R_{\Omega})$, $\dil(f)=\displaystyle\sup_{x\in\partial\Omega}|df(x)|$ (for more details on the Lipschitz distance see \cite{Gro}).\end{prooflipschitz}

\

We end this section by the construction of simple examples that prove the sharpness of Theorem \ref{Lipschitz} with respect of the power of delta involved in our estimates:

The sharpness in the case $n=3$ is already contained in Fuglede's work \cite{Fug1}. In the case $n\geqslant 3$, let $\varphi:\R^n\to\R$ the function defined by
\begin{equation}
\varphi(x)=\left\{\begin{array}{ll}0&\mbox{ if }|x|\geqslant r:=\delta^\frac{p}{2p-2n+pn}\\
\frac{1}{3}(r-|x|)^{2-\frac{n}{p}}&\mbox{ if }\frac{r}{2}\leqslant |x|\leqslant r\\
\frac{1}{3}(2(\frac{r}{2})^{2-\frac{n}{p}}-|x|^{2-\frac{n}{p}})&\mbox{ if }|x|\leqslant r/2
\end{array}\right.
\end{equation}
$\varphi$ is a $C^{1,1-\frac{n}{p}}$ function on $\R^n$ with 
\begin{equation}
\nabla\varphi(x)=\left\{\begin{array}{ll}0&\mbox{ if }|x|\geqslant r\\
\frac{1}{3}(\frac{n}{p}-2)(r-|x|)^{1-\frac{n}{p}}\frac{x}{|x|}&\mbox{ if }\frac{r}{2}\leqslant |x|\leqslant r\\
\frac{1}{3}(\frac{n}{p}-2)|x|^{1-\frac{n}{p}}\frac{x}{|x|}&\mbox{ if }|x|\leqslant r/2
\end{array}\right.
\end{equation}
from which we infer that $\|\varphi\|_\infty= C(n,p)\delta^\frac{2p-n}{2p-2n+pn}$, $\|d\varphi\|_\infty\leqslant C(n,p)\delta^\frac{p-n}{2p-2n+pn}$ and $\frac{1}{C(n,p)}\delta\leqslant\int_{\R^n}|d\varphi|^2\dhn\leqslant C(n,p)\delta$. $\varphi$ can be transposed to a function defined on $\S^n$ (via the exponential map at a fixed point of $\S^n$) for $\delta$ small enough. The previous estimates will be preserved and the surface $S_\varphi=\{(1+\varphi(x))x,x\in\S^n\}$ will be an almost spherical surface in the sense of Fuglede. In particular, according to the inequality (I.a) of \cite{Fug1}, the isoperimetric deficit of the domain $\Omega_\varphi$ bounded by $S_\varphi$ satisfies $\frac{\delta}{C(n,p)}\leqslant\delta(\Omega_\varphi)\leqslant C(p,n)\delta$. Since $d_H(S_\varphi,S_{x_\Omega}(R_\Omega))=\|\varphi\|_\infty$, and for any $q<p$ there exists $K(n,q)$ such that $\|\nabla d\varphi\|_q\leqslant K(n,q)$ for any $\delta>0$, we infer that $\|B\|_q\leqslant K(n,q)$ for any $\delta>0$. These examples prove that the estimate of Theorem \ref{Lipschitz} are sharp with respect to the powers of $\delta$ involved in the estimate on $d_H$. An easy computation show that it is the same way for the estimate on $d_L$.

\section{Almost extremal domains for Chavel's inequality}
\begin{proofchavelpinc}
Let $\Sigma$ be an embedded compact hypersurface bounding a domain $\Omega$ in $\R^{n+1}$ and let $X$ be the vector position. Up to a translation we can assume that $\displaystyle\int_{\Sigma} X\dhn=0$ which allows us to use the variational characterization. Then
\begin{align*}(n+1)^2\frac{|\Omega|^2}{\hn(\Sigma)^2}&=\left(\frac{1}{\hn(\Sigma)}\int_{\Omega}\frac{1}{2}\Delta|X|^2\dhnpi\right)^2=\left(\frac{1}{\hn(\Sigma)}\int_{\Sigma}\scal{X}{\nu}\dhn\right)^2\\
&\leqslant\normlp{X}{1}^2\leqslant\normlp{X}{2}^2\leqslant\frac{\|dX\|_2^2}{\lambda_1^\Sigma}=\frac{n}{\lambda_1^{\Sigma}}\\
&=\frac{(n+1)^2\hn(\Sigma)^{2/n}}{I(\Omega)^{2\left(\frac{n+1}{n}\right)}}(1+\gamma(\Omega))=(n+1)^2\frac{|\Omega|^2}{\hn(\Sigma)^2}(1+\gamma(\Omega))
\end{align*}
Let us put $\rho_{\Omega}:=(n+1)\frac{|\Omega|}{\hn(\Sigma)}$. From the inequalities above we deduce that
$$|\normlp{X}{2}^2-\rho_{\Omega}^2|\leqslant\rho_{\Omega}^2\gamma(\Omega)\ \ \text{and}\ \ |\normlp{X}{1}^2-\rho_{\Omega}^2|\leqslant\rho_{\Omega}^2\gamma(\Omega) $$
which gives for $\gamma(\Omega)<1$
\begin{align*}\normlp{|X|-\rho_{\Omega}}{2}^2&=\normlp{X}{2}^2-2\rho_{\Omega}\normlp{X}{1}+\rho_{\Omega}^2\\
&\leqslant\rho_{\Omega}^2((1+\gamma(\Omega))-2(1-\gamma(\Omega))^{1/2}+1)\leqslant 3\rho_{\Omega}^2\gamma(\Omega)
\end{align*}
Now by the divergence theorem to the field $Z=(|X|-\rho_\Omega)\frac{X}{|X|}$, we get
\begin{align*}
|\Omega\setminus B_0(\rho_{\Omega})|&\leqslant\int_{\Omega\setminus B_0(1)(\rho_{\Omega})}\div Z\dhnpi=\int_{\partial\Omega\setminus B_0(\rho_\Omega)}(|X|-\rho_\Omega)\langle \frac{X}{|X|},\nu\rangle\dhn\\
&\leqslant\hn(\Sigma)\bigl\||X|-\rho_\Omega\bigr\|_1\leqslant 3^{1/2}\hn(\Sigma)\rho_{\Omega}\gamma(\Omega)^{1/2}
\end{align*}
Now since $\rho_{\Omega}=\frac{1}{1+\delta(\Omega)}R_{\Omega}\leqslant R_{\Omega}$ and $|B_0(R_{\Omega})|=|\Omega|$ we have
$$|B_0(\rho_{\Omega})\setminus\Omega|\leqslant|\Omega\setminus B_0(\rho_{\Omega})|\leqslant 3^{1/2}\hn(\Sigma)\rho_{\Omega}\gamma(\Omega)^{1/2}$$
It follows that $\bigl||\Omega|-|B_0(\rho_{\Omega})|\bigr|\leqslant 2 (3^{1/2})\hn(\Sigma)\rho_{\Omega}\gamma(\Omega)^{1/2}$. From the expression of $\rho_{\Omega}$, $I(\Omega)$ and the fact that $I(B_0(1))=(n+1)|B_0(1)|^{\frac{1}{n+1}}$, the last inequality can be rewritten as
$$\left|1-\frac{I(B_0(1))^{n+1}}{I(\Omega)^{n+1}}\right|\leqslant 2(n+1) 3^{1/2}\gamma(\Omega)^{1/2}$$
which gives the desired result.
\end{proofchavelpinc}

\end{document}